\documentclass[12pt,reqno]{amsart}
\usepackage{mathtools}
\usepackage{amssymb, mathabx, enumitem}
\usepackage{color}
\usepackage[colorlinks=true,linkcolor=blue,urlcolor=blue,citecolor=blue, pagebackref]{hyperref}
\usepackage[alphabetic]{amsrefs}
\usepackage{tikz-cd}
\usepackage{tikz}
\usepackage{tikz-3dplot}
\usetikzlibrary{arrows}
\usepackage{graphics}
\usepackage{arydshln}

\usepackage[margin=1in]{geometry}

\newcommand{\C}{\mathbb{C}}
\newcommand{\Q}{\mathbb{Q}}
\newcommand{\Z}{\mathbb{Z}}

\newcommand{\R}{\mathbb{R}}
\newcommand{\op}{\operatorname}

\newtheorem{theorem}{Theorem}[section]

\newtheorem{remark}[theorem]{Remark}
\newtheorem{conjecture}[theorem]{Conjecture}

\newtheorem{corollary}[theorem]{Corollary}

\newtheorem{proposition}[theorem]{Proposition}

\newtheorem{lemma}[theorem]{Lemma}

\newtheorem{definition}[theorem]{Definition}
\newtheorem{definition/lemma}[theorem]{Definition/Lemma}

\title{Weight polytopes and saturation of Demazure characters}
\author{Marc Besson}
\address{BICMR, Peking University, Beijing, P.R. China, 100871}
\email{marc@bicmr.pku.edu.cn}
\author{Sam Jeralds}
\address{University of Queensland, St. Lucia, QLD 4067}
\email{s.jeralds@uq.edu.au}
\author{Joshua Kiers}
\address{Marian University, 3200 Cold Spring Rd, Indianapolis, IN 46222}
\email{jkiers@marian.edu}

\begin{document}
\maketitle
\begin{abstract}
For $G$ a reductive group and $T\subset B$ a maximal torus and Borel subgroup, Demazure modules are certain $B$-submodules, indexed by elements of the Weyl group, of the finite irreducible representations of $G$. In order to describe the $T$-weight spaces that appear in a Demazure module, we study the convex hull of these weights -- the Demazure polytope. We characterize these polytopes both by vertices and by inequalities, and we use these results to prove that Demazure characters are saturated, in the case that $G$ is simple of classical Lie type. Specializing to $G=GL_n$, we recover results of Fink, M\'esz\'aros, and St. Dizier, and separately Fan and Guo, on key polynomials, originally conjectured by Monical, Tokcan, and Yong. 
\end{abstract}

\section{Introduction}

Let $G$ be a complex, reductive algebraic group, with fixed maximal torus and Borel subgroup $T$ and $B$, respectively. The finite-dimensional irreducible representations of $G$ are parametrized by the dominant integral weights $\lambda$. Each such representation $V_\lambda$ is also a representation of $B$, indecomposable but not irreducible. A special class of $B$-submodules, first introduced by Demazure \cite{Dem1, Dem2}, are indexed by Weyl group elements $w\in W$ as follows. First, let $v_\lambda$ denote any nonzero highest-weight vector in $V_\lambda$. Then define the \emph{Demazure module}
$$
V_\lambda^w:=\text{the smallest $B$-submodule of $V_\lambda$ containing $wv_\lambda$};
$$
we give an explicit algebraic construction of these modules in Section \ref{DemMod}. When $w=w_0$, the longest element in $W$, one recovers $V_\lambda^{w_0} = V_\lambda$ (as $B$-modules). 

A natural question is how the $B$-module $V_\lambda^w$ restricts as a $T$-module. Specifically, for any character $\mu:T\to \C^*$, define 
$$
V_\lambda^w(\mu):=\{v\in V_\lambda^w : tv = \mu(t)v~\forall t\in T\};
$$
then the questions at hand are: 
\begin{enumerate}[label=(\alph*)]
\item for which characters $\mu$ is $V_\lambda^w(\mu)\ne 0$? 
\item if nontrivial, what is the dimension of $V_\lambda^w(\mu)$?
\end{enumerate}
The formal $T$-character of $V_\lambda^w$ records exactly these dimensions via
$$
\op{char}(V_\lambda^w) := \sum_{\mu} c_{\lambda,w}^\mu e^\mu,
$$
where $c_{\lambda,w}^\mu:=\dim V_{\lambda}^w(\mu)$. The Demazure character formula \cite{Anderson} provides a mechanism for calculating this formal character. We recall this precisely in Section \ref{DemMod}. 

The main drawback to using this character formula for answering question (b) or even question (a) is the presence of both positive and negative contributions, which makes it in some sense inefficient. One would prefer a manifestly nonnegative or \emph{combinatorial} formula for $c_{\lambda,w}^\mu$. This has been carried out extensively by numerous authors for 
$G=GL_n$ using the language and techniques of algebraic combinatorics -- for example, in the study of key polynomials (which serve as the $GL_n$ Demazure characters), key tableaux, and Demazure atoms. We refer to \cite{HY,Mas} and the references within for results in these directions. Recent developments include the relationship between question (a) and the determination of which Schubert varieties are ``Levi-spherical'' \cite{GHY}. 

In general type, the work of Littelmann \cite{Lit1, Lit2} and Kashiwara \cite{Kas} on crystal bases, Demazure crystals, and the path model leads to combinatorial constructions which answer questions (a) and (b) for any $G$ and fixed $\lambda$ and $w \in W$. However, these often require specialized computations and are less effective at producing global statements for generic $G$, $\lambda$, and $w$.  

When $w = w_0$, there is a simple answer to question (a), for arbitrary $G$ and $\lambda$. In this case $V_\lambda^{w_0} = V_\lambda$, and a classical result (see, for example, \cite{Hal}*{Theorem 7.41}) says that $V_\lambda(\mu) \neq 0$ if and only if $\lambda - \mu \in Q$ and 
$$
\mu \in \op{conv} (\op{wt}(V_\lambda)) = \op{conv} (\{ v \lambda : v \in W\}),
$$
where $\op{conv}$ denotes the rational convex hull inside $X^*(T)\otimes \Q$, where $X^*(T)$ denotes the lattice of algebraic characters of $T$,
$Q$ denotes the root lattice of $G$ with respect to $T$ (that is, the integral span of the root system $\Phi \subset X^*(T)$), and $\op{wt}(V_\lambda)$ is the set of nontrivial weights of $V_\lambda$ (see Section \ref{Notation} for full notational conventions). 

One of our main goals is to generalize this result to arbitrary $w \in W$ as an answer to question (a). Our main theorem is precisely of this form when $G$ is simple and of classical type (that is, of type $A_r$, $B_r$, $C_r$, or $D_r$) or exceptional type $F_4$ or $G_2$. Before we state our results, we introduce our main object of study: the \emph{weight polytope} $P_\lambda^w$ of a Demazure module $V_\lambda^w$, defined as 
\[P_{\lambda}^w:= \op{conv}(\{\mu \in X^*(T): V_{\lambda}^w(\mu)\neq 0 \}).\]
As a preliminary result (Theorem \ref{convexequal}), we describe $P_\lambda^w$ by its vertices, 
$$
P_\lambda^w = \op{conv}\{v\lambda :v \in W, \  v \leq w\},
$$
generalizing the case for $w=w_0$. We also note that when $G=GL_n$ and $w$ is arbitrary, this vertex description follows from work of Fan and Guo on Schubitopes \cite{FG}, answering a conjecture of Monical, Tokcan, and Yong \cite{MTY}*{Conjecture 3.13} on key polynomials. With this in hand, we can now state the primary result of this paper (c.f. Theorems \ref{A-result}, \ref{BCD-result}, and Section \ref{exceptionals}). 

\begin{theorem}\label{lattice}
Suppose that $G$ is of type $A_r$, $B_r$, $C_r$, $D_r$, $F_4$ or $G_2$. Then $V_\lambda^w(\mu)\ne0$ if and only if $\lambda-\mu\in Q$ and $\mu\in P_\lambda^w$. 
\end{theorem}

\begin{remark}
	We conjecture that Theorem \ref{lattice} holds in the exceptional types $E_6$, $E_7$ and $E_8$ as well. See Section \ref{exceptionals} for discussion on approaches to the exceptional types.
\end{remark}

In type $A_r$, Theorem \ref{lattice} was conjectured by Monical, Tokcan, and Yong \cite{MTY}*{Conjecture 3.10}, phrased in terms of the Newton polytopes of key polynomials. Fink, M{\'e}sz{\'a}ros, and St. Dizier proved this result in \cite{FMSD}; i.e., they showed that key polynomials have \emph{saturated Newton polytopes} (see \cite{MTY}*{Definition 1.1}).
Since $P_\lambda^w$ is the Newton polytope of the Demazure character $\op{char}(V_\lambda^w)$ in that setting, our Theorem \ref{lattice} recovers this result for $G=GL_n$. The techniques in \cite{FMSD} come from matroid theory and seem distinct from ours. 
In contrast, the present approach 
is uniform in scope and situated naturally in the original geometric and representation-theoretic context of Demazure's work. 
Nevertheless, in analogy with the more combinatorial approach from $GL_n$, we make the following definition for Demazure characters in arbitrary Lie type.

\begin{definition}
We say that a Demazure character $\op{char}(V_\lambda^w)$ is \emph{saturated} if, for any $\mu$ satisfying $\lambda-\mu \in Q$ and $\mu \in P_\lambda^w$, $e^\mu$ appears with nonzero coefficient in $\op{char}(V_\lambda^w)$.
\end{definition}

\noindent
The usage of the term ``saturated" in this 
context is justified by Lemma \ref{borel-weil}. With this language, Theorem \ref{lattice} can be restated as in the following corollary. 

\begin{corollary} \label{restated} 
Suppose that $G$ is of type $A_r$, $B_r$, $C_r$, $D_r$, $F_4$ or $G_2$. Then all Demazure characters $\op{char}(V_\lambda^w)$ are saturated. 
\end{corollary}

To establish Theorem \ref{lattice}, we first initiate an in-depth study of the Demazure weight polytopes $P_\lambda^w$. In particular, in Section \ref{sectionconvexitysetup} we describe $P_\lambda^w$ as the convex hull of a simple set of vertices. We employ techniques of Geometric Invariant Theory (GIT) in Section \ref{GITsection} to derive a sufficient set of inequalities that determine when $\mu \in P_\lambda^w$. These inequalities are strengthened in Section \ref{ineq-comb} to a smaller set of necessary and sufficient inequalities by recasting the geometric results into combinatorial language. This is done via the \emph{Demazure product} on $W$; we recall the definition of this operation and derive the properties we need in Section \ref{ineq-comb}. In Section \ref{faces}, we study the faces of the polytope $P_\lambda^w$. Here, we find (Corollary \ref{facetDemazurestructure1}) that the faces of the polytope are themselves Demazure weight polytopes associated to Levi subgroups of $G$. It would be fruitful to compare these results with the results in \cite{TW} on faces and inequalities of Bruhat interval polytopes. Note that the collection of Bruhat interval polytopes and the collection of type $A$ Demazure polytopes are distinct but overlapping; see Remark \ref{BIP}. 

With all of this, we can begin constructing the proof of Theorem \ref{lattice} and the saturation of Demazure characters. In Section \ref{IntegralPoints} we prove a key result (Proposition \ref{indsketch}) which will be used in an inductive approach to the theorem. Finally, we derive the saturation of Demazure characters for type $A_r$ in Section \ref{sectiontypeA} and the remaining classical types in Section \ref{sectiontypesbcd}. In Section \ref{exceptionals}, we report on our proof of saturation for types $F_4$ and $G_2$ using computer calculations. There we rephrase Theorem \ref{lattice} in terms of the convex cone defined by the inequalities of Section \ref{ineq-comb} and examine the Hilbert basis of the underlying semigroup of interest. We make a connection between the saturation of Demazure characters in any simple type and the smallest possible Hilbert basis (Proposition \ref{PropP}), and conclude with a conjecture concerning the remaining exceptional types. 

\subsection{Connections with cycles in affine Grassmannians}
This work takes place entirely within the geometric framework of line bundles on Schubert varieties, the Demazure character formula, and the associated combinatorics. However, the original impetus for the study of this problem was related to the geometry of cycles in the affine Grassmanian, Mirkovi\'{c}-Vilonen cycles and Mirkovi\'{c}-Vilonen polytopes. We briefly mention the connections to these topics in this subsection, see \cite{MV} and \cite{Kamn} for notation and background. 

Recall that $\lambda \in X^*(T)$ for $T \subset G$; we can identify $\lambda$ as a coweight $\lambda \in X_*(T^{\vee})$ for $T^{\vee} \subset G^{\vee}$, the Langlands dual group over $\mathbb{C}$. Let $\mathrm{Gr}_{G^{\vee}}$ denote the affine Grassmannian for $G^{\vee}$, $T_{\mu}$ the (opposite) semi-infinite cell  $N^{-}_{\mathcal{K}} \cdot L_\mu$ associated to $\mu \in X^*(T)=X_*(T^{\vee})$, and $\mathrm{Gr}_{G^{\vee}}^{\lambda}$, $\overline{\mathrm{Gr}}_{G^{\vee}}^{\lambda}$ the affine Schubert cell and affine Schubert varieties associated with $\lambda \in X_*(T^{\vee})$.

In their proof of the geometric Satake isomorphism \cite{MV}, Mirkovi\'{c} and Vilonen introduced certain cycles, now known as Mirkovi\'{c}-Vilonen cycles (or MV cycles for short), as the irreducible components of $\overline{\mathrm{Gr}}_{G^{\vee}}^{\lambda} \cap T_{\mu}$, as long as this intersection is nonempty. In \cite{Kamn}, (see also  \cite{JAnderson}), Kamnitzer assigned to each such irreducible component (MV cycle) a pseudo-Weyl polytope (MV polytope) and performed a careful study of the associated combinatorics of these polytopes and their relationship to crystal bases.

The collection of MV polytopes (which are the moment map images of MV cycles) is a broad class of lattice polytopes. The easiest example is the (unique) irreducible component $C$ of $T_{w_0\lambda} \cap \overline{\mathrm{Gr}}_{G^{\vee}}^{\lambda}$; the associated MV polytope is exactly the Weyl polytope $W_{\lambda}=P_{\lambda}^{w_0}$. Subsequently, the work of Naito and Sagaki \cite{Naito}*{Theorem 4.1.5} showed that in fact arbitrary Demazure weight polytopes $P_\lambda^w$ arise as ``extremal" MV polytopes. 

\begin{corollary}
	The Demazure polytopes $P_{\lambda}^w$ admit the following descriptions. They are (1) convex hulls of the extremal weights of the Demazure module $V_{\lambda}^w$ (Theorem \ref{convexequal}) and (2) the extremal MV polytopes of weight $w\lambda$.  
	\end{corollary}


We connect the present work to \cite{Naito} by giving descriptions of the $P_{\lambda}^w$ as pseudo-Weyl polytopes in Section \ref{pseudo}; for more on Demazure polytopes as MV polytopes and cycles in the affine Grassmannian, we refer the reader to \cite{Dykes}.

Schwer  \cite{Schwer}*{Theorem 11.7} provides an interpretation on the geometric side, showing that the dimension of each weight space of a Demazure module (for $B)$ is equal to the number of irreducible components of ``maximal" dimension of $IL_{w\lambda} \cap T_{\mu}$ in $\overline{\mathrm{Gr}}_{G^{\vee}}$, where $I$ is the Iwahori subgroup associated to $B^{\vee}$ (see also Ion \cite{Ion}, Theorem 1 for a related statement).
As Theorem \ref{lattice} exactly parametrizes the nontrivial weight spaces of $V_\lambda^w$ via lattice points in $P_\lambda^w$ (apart from simple factors of type $E$), we arrive at the following connection. 

 
 \begin{corollary}
	The integral lattice points $\nu \in P_{\lambda}^w$ admit the following descriptions. They (1) parametrize nontrivial weight spaces $V_{\lambda}^{w}(\nu) \neq 0$ in the Demazure module $V_{\lambda}^{w}$ and (2) they correspond to points $L_{\nu} \in \mathrm{Gr}_{G^{\vee}}$ such that the intersection $IL_{w \lambda} \cap T_{\nu}$ has a component of ``maximal'' dimension. 
\end{corollary}


While the above geometric incarnations via generalized MV cycles are perfectly appropriate, it is understandable to seek an even more straightforward geometric interpretation. For example, for a Demazure polytope $P_{\lambda}^w$, do the integral lattice points inside the polytope correspond to points $L_{\mu}$ which lie in the closure of $IL_{w \lambda} \cap T_{\lambda}$? This is clearly true in the two extreme cases ($w=e$ and $w=w_0$), and is not too difficult to see for $l(w)=1.$ For other $w$ this interpretation has not yet been established.

\subsection{Acknowledgements}

The authors are very grateful to D. Anderson and 
A. Yong for their feedback on an earlier version of the manuscript. The authors also wish to thank J. Kamnitzer for questions and comments on the first draft.
The first author wishes to thank J. Hong for discussions on this problem during the course of their studies of affine Schubert varieties \cite{BH}. Finally, we thank the anonymous referee for many helpful suggestions, comments, and corrections. 

\section{Notation} \label{Notation}

We fix $G$ a reductive linear algebraic group over $\mathbb{C}$. We choose a maximal torus and Borel subgroup $T \subset B \subset G$, with corresponding Lie algebras $\mathfrak{h}$, $\mathfrak{b}$ and $\mathfrak{g}$.  We denote by $X^*(T)$ the lattice of weights of $T$, and by $X_*(T)$ the lattice of coweights. Their natural pairing is denoted by $\langle \, , \rangle$, e.g. $\langle \lambda, x \rangle$ for $\lambda \in X^*(T)$ and $x \in X_*(T)$. 

We let $\Phi$ denote the set of  roots of $G$ with respect to $T$, and denote by $\Phi^+$ the set of positive roots of $G$ with respect to $B$.  We let $\Phi^{\vee}$ denote the set of coroots, so $(\Phi, X^*(T), \Phi^{\vee}, X_*(T))$ is a root datum for $G$. We write $X^*(T)^+$ for the set of dominant weights and $W=N_G(T)/T$ for the Weyl group of $G$. We denote by $\Delta=\{\alpha_1, \dots \alpha_r\}$ (resp., $\{\alpha^{\vee}_1, \dots \alpha^{\vee}_r\} $) the set of simple roots in  $\Phi$ (resp., simple coroots in $\Phi^{\vee}$). The fundamental coweights $x_i\in \op{span}_\Q(\{\alpha_1^\vee,\hdots,\alpha_r^\vee\})$ are uniquely defined by the requirement that $\langle \alpha_i, x_j \rangle = \delta_{ij}$. 

For any subset $S\subset\Delta$ there is a standard parabolic subgroup $P_S$ of $G$ generated by $B$ and the negative root subgroups $U_{-\alpha_i}$ for $\alpha_i\in S$. The Weyl group $W_P$ of $P$ is by definition the subgroup of $W$ generated by the simple reflections $s_i$ for $\alpha_i\in S$. Each coset in $W/W_P$ has a unique minimal-length representative. The set of all such representatives is denoted by $W^P$. If $S = \{1,\hdots, r\}\setminus \{i\}$, then we usually use $P_i$ for the maximal parabolic $P_S$.  

If $\lambda \in X^*(T)^+$, we write $V_{\lambda}$ for the irreducible representation of $G$ of highest weight $\lambda$. We write $V_{\lambda}^w$ for the Demazure module associated to $w$ and $\lambda$. If $\mu \in X^*(T)$ and $V$ is a $T$-module, we write $V(\mu)$ for the subspace of weight $\mu$.

Let $\{e_{\alpha_i}, h_{\alpha_i}, f_{\alpha_i}\}$ ($1 \leq i \leq r$) be a Chevalley basis for $\mathfrak{g}$. We write $\mathfrak{sl}_{2, \alpha_i}$ for the image of the embedding $\mathfrak{sl}_2 \rightarrow \mathfrak{g}$ associated to $\{e_{\alpha_i}, h_{\alpha_i}, f_{\alpha_i}\}$.
We may also consider the group counterpart of this, as in \cite{Springer}: we write $G_{\alpha_i}$ for the group generated by $U_{\alpha_i}$ and $U_{-\alpha_i}$. Then the ``root subgroup'' $G_{\alpha_i}$ is isomorphic to $SL_{2}(\mathbb{C})$ or $PGL_2(\mathbb{C})$ and $\op{Lie}(G_{\alpha_i})=\mathfrak{sl}_{2,\alpha_i}$.

For many of our polytope calculations we prefer to work rationally, so we write $X^*(T)_\Q$ for $X^*(T)\otimes \Q$. 
For an $n-$dimensional polytope $P$, we call a face of dimension $n-1$ a facet.

\section{Recollections on Demazure modules} \label{DemMod}

In this section, briefly we review the fundamental objects of interest for this paper; namely, the Demazure modules and the Demazure character formula. The study of these modules and their characters has a long history, going back to Demazure \cite{Dem1, Dem2}, with wide-spread applications and perspectives coming from geometry, representation theory, and combinatorics. In this paper, we focus on the classical representation-theoretic aspects of the subject, taking \cite{Anderson} and \cite{KumarBook}*{Ch. 8} as our primary sources. 

Let $\lambda \in X^\ast(T)^+$ be an integral dominant weight, and $V_\lambda$ the associated irreducible highest-weight representation of the Lie algebra $\mathfrak{g}$. For any $w \in W$, there is the one-dimensional extremal weight space $V_\lambda (w\lambda)$; fix a nonzero vector $v_{w\lambda} \in V_\lambda (w\lambda)$. Then the \emph{Demazure module} $V_\lambda^w$ is defined by 
$$
V_\lambda^w := U(\mathfrak{b}).v_{w\lambda},
$$
where $U(\mathfrak{b})$ is the universal enveloping algebra of $\mathfrak{b}$. While this gives an algebraic construction of $V_\lambda^w$, we also briefly sketch a geometric construction, which was the original interest in these modules and will play a major role in the geometric arguments of Section \ref{GITsection}. 

Let $X=G/B$, the flag variety associated to $G$ and the fixed Borel subgroup $B$, and $X_w  = \overline{BwB/B} \subset X$ the Schubert variety associated to $w \in W$. For an integral dominant weight $\lambda \in X^*(T)^+$, let $L(\lambda)$ denote the associated line bundle $G\times_B \C_{-\lambda}$ on $X$ and $L_w(\lambda)$ the pull-back of $L(\lambda)$ to $X_w$. Then, the Demazure module $V_\lambda^w$ can be constructed geometrically via 
$$
V_\lambda^w \cong H^0(X_w, L_w(\lambda))^\ast.
$$
From this description, one can derive the following standard results; we defer to \cite{Anderson},\cite{KumarBook}*{Ch. 8} for full details. 

\begin{proposition} \label{DemProps}
\begin{enumerate} 
\item For any $v \leq w$ in the Bruhat order, the canonical restriction map $H^0(X_w, L_w(\lambda)) \to H^0(X_v, L_v(\lambda))$ is surjective, and induces an inclusion $V_\lambda^v \hookrightarrow V_\lambda^w$. 
\item Let $w=s_{i_1}\dots s_{i_k}$ be a reduced expression. Let $G_{\alpha_{i_1}} \subset G$ be the root subgroup associated to $\alpha_{i_1}$. Then $X_w$ is $G_{\alpha_{i_1}}$-stable, so both $H^0(X_w, L_w(\lambda))$ and $V_\lambda^w$ are $G_{\alpha_{i_1}}$-modules. 
\end{enumerate}
\end{proposition}

To compute the character of $V_\lambda^w$ (as a $T$-module), we recall the \emph{Demazure operators} $D_i$, indexed by the simple roots $\alpha_i$.

\begin{definition} For arbitrary $\lambda \in X^\ast(T)$ and simple root $\alpha_i$, define
	
	  \[
	D_i(e^{\lambda}) = \frac{1-e^{-\alpha_i}s_i}{1-e^{-\alpha_i}}e^\lambda =  \begin{cases}
	e^{\lambda}+e^{\lambda-\alpha_i} \dots + e^{s_{i}\lambda}, & \text{for } \langle \lambda, \alpha_i^{\vee} \rangle \geq 0, \\
	0, & \text{for } \langle \lambda, \alpha_i^{\vee} \rangle =-1,\\
	-(e^{\lambda+ \alpha_i}+ \dots + e^{s_{i}\lambda-\alpha_i}), & \text{for } \langle \lambda, \alpha_i^{\vee} \rangle <-1.
	\end{cases} 
	\]

We extend by linearity to obtain operators on $\mathbb{Z}[X^*(T)]$.
\end{definition}
\noindent The following easy lemma is an immediate consequence of the definition of $D_i$ and is a first indication of the connection between the Demazure operators and characters. 

\begin{lemma} \label{virtchar}
	Let $S= \sum m_\lambda e^{\lambda} \in \mathbb{Z}[X^*(T)]$. Then $D_i(S)$ is the character of a virtual representation of $\mathfrak{sl}_{2, \alpha_{i}}$.
\end{lemma}

\noindent Working from the definition, it is easy to verify that for any $S \in \mathbb{Z}[X^*(T)]$, $D_i \circ D_i (S)= D_i (S)$. Less obvious, but well-known, is that the Demazure operators satisfy the braid relations for the associated Weyl group $W$. Thus, for any reduced word $w=s_{i_1} s_{i_2} \dots s_{i_k}$, there is a well-defined operator $D_w$, via 
$$
D_w:= D_{i_1} D_{i_2} \cdots D_{i_k}.
$$
With this in hand, the \emph{Demazure character formula} is the following connection between $D_w$ and $V_\lambda^w$; see, for example, \cite{Anderson}. 

\begin{theorem}\label{dem-char}
	Let $\lambda \in X^*(T)^+$ and $w=s_{i_1} \dots s_{i_k}$ be a reduced word. Then \[
	\op{char} V_\lambda^w =D_w (e^\lambda)= D_{i_1} \dots D_{i_k}(e^{\lambda})\]
\end{theorem}

Finally, using Proposition \ref{DemProps}, we get the following strengthening of Lemma \ref{virtchar} which will be used throughout what follows. 

\begin{corollary}  \label{sl2-stable} 
For $\lambda \in X^*(T)^+$, and for $s_{i_1} \dots s_{i_k}$ a reduced word for $w \in W$, $\op{char}(V_{\lambda}^w)$ is an honest (not virtual) character of an $\mathfrak{sl}_{2,\alpha_{i_1}}$-module.
\end{corollary}

\section{Convex Hulls}\label{sectionconvexitysetup}
In this section we establish some basic results on the convex hulls of the weights appearing nontrivially in Demazure characters. Fix a dominant weight $\lambda$ and a $w\in W$. 
	Let $\op{wt}(V_{\lambda}^w)$ denote the set of weights $\mu$ such that $V_{\lambda}^w(\mu) \neq 0$. 
We consider the convex hulls $P_{\lambda}^w := \op{conv}(\op{wt}(V_{\lambda}^w)) \subset X^*(T)_\Q$ and $P_{w}^\le:= \op{conv}(\{u\lambda \, : \,  u \leq w\}) \subset X^*(T)_\Q$. Note that, by Proposition \ref{DemProps}(1), we get an immediate containment $P_{w}^\le \subseteq P_{\lambda}^w$. Our aim is to show that in fact $P_\lambda^w = P_{w}^\le$. 

Following the notation of \cite{BGG}, for $w\in W$ and $s_\gamma$ a root reflection, we use $s_\gamma w \to w$ to mean that $s_\gamma w \le w$ and $\ell(s_\gamma w) =  \ell(w)-1$. 
We will need the following standard result, which is a corollary of \cite{BGG}*{Lemma 2.5}:

\begin{lemma}\label{pqrs}
Suppose that $s_i w \to w$. Then $s_{i}\{u \, : \,  u \leq w\}=\{u \, : \,  u \leq w\}$. 
\end{lemma}

Let $s_{i_1} \dots s_{i_k}$ be a reduced word for $w$, which we fix for the remainder of this section. 

\begin{lemma} 
  The $s_{i_1}$-action on $X^*(T)_\Q$ induces an order-two automorphism on both $P_{\lambda}^w$ and $P_{w}^\le$. 
\end{lemma}

\begin{proof}
	If a finite set $P \subset V$ is stable under a linear map $f:V \rightarrow V$, then so is $\op{conv}(P)$. Since $s_{i_1}\cdot \op{wt}(V_{\lambda}^w)=\op{wt}(V_{\lambda}^w)$ by Corollary \ref{sl2-stable}, $P_\lambda^w$ is the convex hull of an $s_{i_1}$-stable finite set. Likewise, since $s_{i_1}\{u\lambda \, : \,  u \leq w\}=\{u\lambda \, : \,  u \leq w\}$ by Lemma \ref{pqrs}, $P_{w}^\le$ is also the convex hull of an $s_{i_1}$-stable finite set. 
\end{proof}

Now consider an arbitrary (not necessarily integral) weight $\mu \in X^*(T)_\Q$ and the line $\mu+\mathbb{Q}\alpha_{i_1}$. Let $I=\{k: \mu+ k\alpha_{i_1}\in P_{w}^\le\}\subseteq \Q$ be the interval parametrizing the intersection $(\mu+\mathbb{Q}\alpha_{i_1})\cap P_{w}^\le$; say $I=[k'',k']$ and set $\mu' = \mu+k' \alpha_{i_1}$ and $\mu'' = \mu+k'' \alpha_{i_1}$. Thus $(\mu+\mathbb{Q}\alpha_{i_1})\cap P_{w}^\le$ is the line segment between $\mu'$ and $\mu''$. See Figure \ref{hey1}.
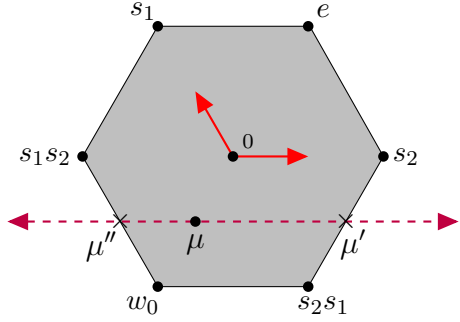
\begin{figure}
\begin{center}
\begin{tikzpicture}[>=triangle 45]

\draw[fill=lightgray] (2.0,0) -- (1.02,1.73205) -- (-1.0, 1.73205) -- (-2.0,0) -- (-1.0,-1.73205) -- (1.0,-1.73205) -- (2.0,0);
\draw[thick,red,->] (0,0) -- (1,0);
\draw[thick,red,->] (0,0) -- (-0.5,0.866025);

\draw[thick,dashed,purple,<->] (-3,-0.866025) -- (3,-0.866025);
\fill (-0.5,-0.866025) circle [radius=0.07];
\fill (0,0) circle [radius=0.07];

\fill (1,1.73205) circle [radius=0.07];
\fill (-1,1.73205) circle [radius=0.07];
\fill (1,-1.73205) circle [radius=0.07];
\fill (-1,-1.73205) circle [radius=0.07];
\fill (2,0) circle [radius=0.07];
\fill (-2,0) circle [radius=0.07];
\node at (-1.5,-0.866025) {$\times$};
\node at (1.5,-0.866025) {$\times$};
\node at (1.2,1.93205) {\small $e$};
\node at (-1.2,1.93205) {\small $s_1$};
\node at (1.2,-1.98205) {\small $s_2s_1$};
\node at (-1.2,-1.98205) {\small $w_0$};
\node at (2.3,0) {\small $s_2$};
\node at (-2.5,0) {\small $s_1s_2$};

\node at (0.2,0.2) {\tiny $0$};
\node at (-0.5,-1.166025) {$\mu$};

\node at (1.6,-1.166025) {$\mu'$};
\node at (-1.75,-1.196025) {$\mu''$};

\end{tikzpicture}
\caption{\label{hey1} \small In type $A_2$, the Weyl polytope $P_\lambda^{s_1s_2s_1}$. Here $\lambda = 2\omega_1+2\omega_2$. The red arrows depict simple roots. }
\end{center}
\end{figure}

\begin{corollary}\label{min-0-max}
	Assume that $(\mu+\mathbb{Q}\alpha_{i_1} )\cap P_{w}^\le$ is nonempty. Then $s_{i_1} \mu'=\mu''$. So if $\mu'=\mu''$ then $s_{i_1}\mu'=\mu'$.
	If $\mu \in P_{w}^\le$, then $k'' \leq 0 \leq k'$.
\end{corollary}

\begin{proof}
Since $P_\lambda^w$ and $\mu + \Q \alpha_{i_1}$ are both $s_{i_1}$-stable, so is their intersection. Therefore $s_{i_1}$ interchanges the endpoints.  
Moreover, if $\mu\in P_\lambda^w$, then $\mu$ is contained in this interval; hence $0\in I$. 
\end{proof}

\begin{remark}
	The previous corollary also applies to $P_{\lambda}^w$ since it only relies on $s_{i_1}P_{w}^\le=P_{w}^\le$, which is also true for $P_{\lambda}^w$.
\end{remark}

We also point out that if $\mu'\in P_{w}^\le$ is maximal in the interval $(\mu'+\Q\alpha_{i_1})\cap P_{w}^\le$, then $\mu'$ is a convex combination of a (usually strict) subset of the vertices of $P_{w}^\le$. In fact, this result is independent of the fact that $s_{i_1}w \to w$. 

\begin{lemma}\label{maximalpoint}
Suppose $\mu' \in P_{w}^\le$ is maximal with respect to the $\alpha_{i_1}$ direction: $\mu'+k\alpha_{i_1}\in P_{w}^\le \implies k\le 0$. Then 
$$
\mu' = \sum_{\tiny \begin{array}{c}u\le w\\ u^{-1}\alpha_{i_1}\succ 0\end{array}} b_u u\lambda,
$$
where each $b_u\in \Q_{\ge 0}$ and $\sum b_u = 1$. 
\end{lemma}

\begin{proof}
Start by writing $\mu' = \sum_{u\le w} b_u u\lambda$, where each $b_u\ge 0$ and $\sum b_u = 1$. Let $u\le w$ be arbitrary such that $u^{-1}\alpha_{i_1}\prec 0$. We argue that $b_u$ is either equal to $0$ or can be set to $0$ by adjusting $b_{s_{i_1}u}$. 

Set $v=s_{i_1}u$. By hypothesis on $u$, $v \xrightarrow{\alpha_{i_1}} u$; hence $v\le w$. Examining the contributions of $u$ and $v$ to the convex combination, we find 
\begin{align*}
\mu' &= b_{v} v \lambda + b_{u} \left(s_{i_1} v\lambda\right)+\sum_{\tiny \begin{array}{c}x\le w\\ x\not\in \{u,v\}\end{array}} b_x x\lambda \\
&= b_{v} v \lambda + b_{u} \left( v\lambda - \langle \lambda, v^{-1}\alpha_{i_1}^\vee \rangle \alpha_{i_1}\right)+\sum_{\tiny \begin{array}{c}x\le w\\ x\not\in \{u,v\}\end{array}} b_x x\lambda \\
&=(b_v+b_u)v\lambda - (b_u \langle \lambda, v^{-1}\alpha_{i_1}^\vee \rangle) \alpha_{i_1}+\sum_{\tiny \begin{array}{c}x\le w\\ x\not\in \{u,v\}\end{array}} b_x x\lambda .
\end{align*}
The last line shows that $\mu' + (b_u \langle \lambda, v^{-1}\alpha_{i_1}^\vee \rangle) \alpha_{i_1}$ is once again a convex combination of vertices of $P_{w}^\le$, hence is contained in $P_{w}^\le$. By maximality of $\mu'$, this forces $b_u \langle \lambda, v^{-1}\alpha_{i_1}^\vee \rangle\le 0$. 

On the other hand, $b_u\ge 0$ and $v^{-1}\alpha_{i_1} \succ 0$, so $b_u\langle \lambda, v^{-1}\alpha_{i_1}^\vee\rangle \ge 0$. Therefore $b_u \langle \lambda, v^{-1}\alpha_{i_1}^\vee \rangle$ must equal $0$, so either 
\begin{enumerate}
\item $b_u=0$, in which case no further action is needed, or
\item  $\langle \lambda, v^{-1}\alpha_{i_1}\rangle = 0$, in which case $v\lambda = s_{i_1}v\lambda = u\lambda$. So $b_v v\lambda + b_u u\lambda = (b_v + b_u) v \lambda + 0 u \lambda$, and we can therefore assume $b_u = 0$.  \qedhere
\end{enumerate}
\end{proof}

Now we reinstate the assumption that $s_{i_1}w\to w$, which was not required for the proof of Lemma \ref{maximalpoint}. 
This provides a natural-seeming corollary: the maximum $\mu'$ does not change from $P_{s_{i_1}w}^\le $ to $P_w^\le$. See Figure \ref{hey2}. 

\begin{corollary}\label{maxs-are-equal}
Keeping the hypotheses of the previous lemma, assume also that $s_{i_1}w\to w$.  Then $\mu'$ belongs to $P_{s_{i_1}w}^\le$, and thus 
$$
\op{max}((\mu'+\Q\alpha_{i_1})\cap P_{w}^\le) = \mu' = \op{max}((\mu'+\Q\alpha_{i_1})\cap P_{s_{i_1}w}^\le). 
$$
\end{corollary}

\begin{proof}
Suppose $u\le w$ and $u^{-1}\alpha_{i_1}\succ 0$. Then $u\to s_{i_1}u$, so in particular $u\ne w$. By \cite{BGG}*{Lemma 2.5}, 
$u\le s_{i_1}w $ or $s_{i_1}u< s_{i_1}w$. Either way, $u\le s_{i_1}w$. So $\mu'$ is in fact a convex combination of vertices of $P_{s_{i_1}w}^\le$. 
\end{proof}

\begin{figure}
\begin{center}
\begin{tikzpicture}[>=triangle 45]

\draw[thick,dashed] (2.02,0) -- (1.02,1.74205) -- (-1.02, 1.74205) -- (-2.02,0) -- (-1.02,-1.74205) -- (1.02,-1.74205) -- (2.02,0);
\draw[fill=lightgray] (2,0) -- (1,1.73205) -- (-1, 1.73205)  -- (1,-1.73205) -- (2,0);

\draw[thick,dashed,purple,-] (-1.5,-0.866025) -- (1.5,-0.866025);
\fill (-0.5,-0.866025) circle [radius=0.07];

\fill (1,1.73205) circle [radius=0.07];
\fill (-1,1.73205) circle [radius=0.07];
\fill (1,-1.73205) circle [radius=0.07];
\fill (-1,-1.73205) circle [radius=0.07];
\fill (2,0) circle [radius=0.07];
\fill (-2,0) circle [radius=0.07];

\node at (0.5,-0.866025) {$\times$};
\node at (1.5,-0.866025) {$\times$};
\node at (-1.5,-0.866025) {$\times$};
\node at (1.2,1.93205) {\small $e$};
\node at (-1.2,1.93205) {\small $s_1$};
\node at (1.2,-1.98205) {\small $s_2s_1$};
\node at (-1.2,-1.98205) {\small $w_0$};
\node at (2.3,0) {\small $s_2$};
\node at (-2.5,0) {\small $s_1s_2$};

\node at (-0.5,-1.166025) {$\mu$};
\node at (1.8,-0.866025) {$\mu'$};
\node at (0.9,-0.566025) {\tiny $\mu''(v)$};
\node at (-1.1,-0.566025) {\tiny $\mu''(w)$};
\end{tikzpicture}
\caption{\label{hey2} \small In type $A_2$, two polytopes $P_\lambda^{v}$ (shaded) and $P_\lambda^{w}$ (dashed outline), where $w = s_1s_2s_1$ and $v = s_2s_1 = s_1w$. The two $\mu''$s differ but $\mu'$ is common to both polytopes. }
\end{center}
\end{figure}
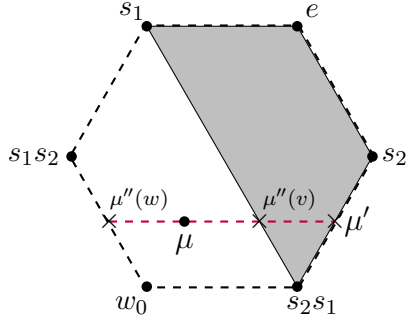

We now show that the polytopes $P_\lambda^w$ and $P_{w}^\le$ coincide. 

\begin{theorem}\label{convexequal}
	We have an equality of polytopes $P_{\lambda}^w=P_{w}^\le$.
\end{theorem}

\begin{proof} 

As before, we have the containment $P_{w}^\le \subseteq P_{\lambda}^w$ by Proposition \ref{DemProps}(1). We now show the reverse inclusion; for this, it suffices to show that $\mu\in \op{wt}(V_\lambda^w)\implies \mu \in P_{w}^\le$. We use induction on $\ell(w)$. 
	
	We have $P_{w}^\le=P_{\lambda}^w$ for all words $w$ of length 0 and length 1.
	
	Assume it is true for all words of length $\leq n-1$, and 
	let $w$ be of length $n$. Assume that $\mu \in \op{wt}(V_{\lambda}^w)$. Write $w=s_{i_1} \cdots s_{i_n}$. If $\mu \in \op{wt}(V_{\lambda}^{s_{i_2} \cdots s_{i_n}})$, by the inductive assumption, then $\mu \in P_{s_{i_2} \cdots s_{i_n}}^\le \subset P_{w}^\le$. Otherwise, we define $\mu'=\mu+k'\alpha_{i_1}$ to be the maximal intersection with  $P_\lambda^{s_{i_2}\cdots s_{i_n}}$ as above (such an intersection exists directly by $D_{i_1} \op{char} V_\lambda^{s_{i_2}\cdots s_{i_n}} = \op{char}V_\lambda^w$).

	By inductive hypothesis, $P_\lambda^{s_{i_2}\cdots s_{i_n}} = P_{s_{i_2}\cdots s_{i_n}}^\le$, so $\mu'$ is also $\op{max}((\mu+\Q \alpha_{i_1}) \cap P_{s_{i_2} \cdots s_{i_n}}^\le)$.  By Corollary \ref{maxs-are-equal}, $\mu'$ is also $\op{max}((\mu+\Q \alpha_{i_1}) \cap P_{w}^\le)$ (see Figure \ref{hey2}). Then Corollary \ref{min-0-max} implies that $s_{i_1}\mu' \le \mu \le \mu'$. Since $\mu$ is on the line segment joining $\mu'$ and $s_{i_1}\mu'$, $\mu \in P_{w}^\le$. 
\end{proof}

Now that we have established that there is no difference between the convex hull of the weights of a Demazure module and the convex hull of the extremal weights, we will drop the notation $P_{w}^\le$ and always use $P_{\lambda}^w$. In what follows, we will use Geometric Invariant Theory to deduce sufficient inequalities for $P_\lambda^w$. Moreover, some straightforward combinatorics yields necessary conditions on $P_{w}^\le$. 
Thus $P_\lambda^w = P_{w}^\le$ produces necessary and sufficient conditions determining the Demazure polytope $P^{w}_{\lambda}$. 

\begin{remark}
As an immediate corollary of the agreement of $P_\lambda^w$ and $P_{w}^\le$, we obtain that the vertices of $P_\lambda^w$ are precisely $\{u\lambda : u\le w\}$. This was conjectured in type $A_r$ by Monical, Tokcan, and Yong \cite{MTY}*{Conjecture 3.13}, and first proven (in type $A_r$) by Fan and Guo \cite{FG}.
\end{remark}

\section{Geometric setup for inequalities} 
\label{GITsection}

In this section, we follow an established tradition (see \cite{Kly,BK,RessI} or the survey \cite{Kumar}) of using Geometric Invariant Theory to impose constraints on representation-theoretic nonvanishing problems. After we review the relevant results from GIT, we indicate how to apply them in the setting of Schubert varieties equipped with the natural $T$-action. The result is a series of linear inequalities (depending on $w$) on weights $\lambda,\mu$ that, when satisfied, suffice to ensure that $\mu\in P_\lambda^w$. 

Suppose $X$ is a projective, normal $G$-variety, where $G$ is a reductive group, and suppose $\mathbb{L}$ is a $G$-linearized line bundle on $X$. Geometric Invariant Theory gives criteria for deciding when $H^0(X,\mathbb{L}^{\otimes n})^G\ne 0$ for some $n\ge 1$; these \emph{Hilbert-Mumford} criteria involve one-parameter subgroups (OPS) $\delta:\C^* \to G$, which are a priori infinite in number. Much work has been done, in various settings,  to reduce to a finite set of criteria. 

\subsection{Certain results from GIT}

Only the case $G=T$ is relevant for us, so assume now that $X$ has a $T$-action and $\mathbb{L}$ is $T$-linearized. We will use foundational results of Kempf \cite{Kempf} and Hesselink \cite{H}, which were further developed by Ressayre \cite{RessI}. For ease of exposition, we'll generally follow the notation of \cite{RessI}*{\S 2}. 

\begin{definition}
A point $x\in X$ is called \emph{semistable} with respect to $\mathbb{L}$ if for some $n>0$, there is a section $\sigma\in H^0(X,\mathbb{L}^{\otimes n})^T$ such that $\sigma(x)\ne0$. If $x$ is not semistable, it is called \emph{unstable}. The set of all semistable points of $X$ is denoted $X^{ss}$. 
\end{definition}

\begin{definition}
Let $x\in X$ and $\delta\in X_*(T)$ a OPS. We define an integer $\mu^{\mathbb{L}}(x,\delta)$ as in \cite{GIT}. First, let $x_0$ be the limit point $x_0=\lim_{t\to 0}\delta(t).x$, which exists since $X$ is projective. Via $\delta$, there is an induced $\C^*$-action on the fibre $\mathbb{L}_{x_0}$, which is necessarily given by some integer $r$ such that 
$$
\delta(t).z = t^r z
$$
for $t\in \C^*$, $z\in \mathbb{L}_{x_0}$. Set 
$$
\mu^{\mathbb{L}}(x,\delta):=-r.
$$
\end{definition}

(This definition of $\mu^{\mathbb{L}}(x,\delta)$ differs from Mumford's convention in \cite{GIT, BK} by $-1$, but agrees with \cite{RessI}.) 
The Hilbert-Mumford criterion of \cite{GIT} states that 
$$
x\text{ is semistable w.r.t. }\mathbb{L}\iff \mu^{\mathbb{L}}(x,\delta)\le0 \text{ for every $\delta\in X_*(T)$}.
$$

Let $||\cdot||$ be a 
Euclidean norm on $X_*(T)\otimes \Q$.
For $x\in X$, set 
$$
M^{\mathbb{L}}(x) = \sup_{\delta\in X_*(T),\delta\ne0} \left\{\frac{\mu^{\mathbb{L}}(x,\delta)}{||\delta||}\right\}.
$$
Kempf showed that this supremum is always uniquely attained (with an analogous statement for general $G$). 

\begin{theorem}[\cite{Kempf}]\label{kempf-thm}
Fix $x\in X$. There exists a unique (up to rescaling) $\delta\in X_*(T)$ such that $M^{\mathbb{L}}(x) = \frac{\mu^{\mathbb{L}}(x,\delta)}{||\delta||}$. 
\end{theorem}

Following \cite{RessI}*{Section 2.3.3}, we say a primitive OPS $\delta\in X_*(T)$ is \emph{adapted to $x$ and $\mathbb{L}$} if $M^{\mathbb{L}}(x) = \frac{\mu^{\mathbb{L}}(x,\delta)}{||\delta||}$. By Theorem \ref{kempf-thm}, we get a set map $X\to X_*(T)$ sending $x$ to the unique primitive $\delta$ adapted to $x$. One can ask about the fibres of this set map. 

\begin{definition}[\cite{RessI}*{Section 2.5}]
Let $d\in \R_{>0}$ and $\tau\in X_*(T)$. Set 
$$
S_{d,\tau}^{\mathbb{L}} = \{x\in X: M^{\mathbb{L}}(x) = d, \text{ and $\tau$ is adapted to $x$ and $\mathbb{L}$}\}.
$$
\end{definition}

Note that each $S_{d,\tau}^{\mathbb{L}}$, if nonempty, consists of unstable points of $X$. Clearly there is a disjoint union 
\begin{align}\label{strata}
X = X^{ss}\sqcup \bigsqcup_{d,\tau} S_{d,\tau}^{\mathbb{L}}.
\end{align}

Hesselink proved that (\ref{strata}) is a finite stratification by locally-closed $T$-invariant subsets \cite{H} (actually he showed a similar statement for general $G$).
As a corollary: 

\begin{corollary}\label{open-dense-stratum}
If $X$ is irreducible and has no semistable points, then there is a unique open stratum $S_{d,\tau}^{\mathbb{L}}$. 
\end{corollary}

\subsection{Line bundles on $X_w$}

Now let $X = X_w$ with its left $T$-action. For any two weights $\mu,\lambda\in X^*(T)$, there is a $T$-linearized line bundle $\C_{\mu} \otimes L(\lambda)$ on $G/B$ with total space
$$
G\times_B \C_{-\lambda}
$$
and $T$-action given by 
$$
t.[g,z] := [tg,\mu(t)z].
$$
By the Borel-Weil theorem for a dominant weight $\lambda \in X^*(T)^+$, $H^0(G/B,\C_{\mu}\otimes L(\lambda))$ is identified with the $T$-representation $\C_{\mu}\otimes V_\lambda^*$, i.e., with $(\C_{-\mu} \otimes V_\lambda)^*$. Therefore 
$$
H^0(G/B,\C_{\mu}\otimes L(\lambda))^T\ne0
$$
if and only if $V_\lambda$ possesses a nontrivial $T$-weight space with weight $\mu$. Likewise, restricting to $X_w$, 
$$
H^0(X_w,\C_{\mu}\otimes L_w(\lambda))^T\ne0
$$
if and only if $V_\lambda^w(\mu)\ne 0$.  More to the point, 
\begin{lemma}\label{borel-weil}
There exists an integer $n\ge1$ such that 
$$H^0(X_w, \left(\C_{\mu}\otimes L_w(\lambda)\right)^{\otimes n})^T \ne 0$$
if and only if $\mu\in P_\lambda^w$. 
\end{lemma}

\begin{proof}
($\Rightarrow$)
For such an integer $n$, $n\mu$ is a nontrivial weight in $V_{n\lambda}^w$. Thus $n\mu$ belongs to $P_{n\lambda}^w$. Since $P_{n\lambda}^w = \op{conv}\{nv\lambda:v\le w\}$, $n\mu$ is a convex combination of $\{nv\lambda:v\le w\}$. Thus $\mu$ is a convex combination of $\{v\lambda: v\le w\}$ and $\mu\in P_\lambda^w$.

($\Leftarrow$)
Start with $\mu = \sum_{v\le w}b_v v\lambda$, where $b_v\in \Q^{\ge0}$ and $\sum b_v = 1$. Finding a common denominator, 
$n\mu = \sum_{v\le w} c_v v\lambda$, where $c_v\in \Z^{\ge 0}$ and $\sum c_v =n$. For each $v\le w$, since $\dim V^w_\lambda(v\lambda) = 1$, fix a nonzero section $\sigma_v\in H^0(X_w, \C_{v\lambda}\otimes L_w(\lambda))^T$. The tensor product of these sections produces a nonzero global section 
$$
\bigotimes_{v\le w} \sigma_v^{\otimes c_v} \in H^0(X_w, \C_{n\mu} \otimes L_w(n\lambda))^T,
$$
as desired. 
\end{proof}

\begin{remark} Lemma \ref{borel-weil} is a consequence of more general results due to Brion and Procesi \cite{BP}; we include a proof specific to our setting for the convenience of the reader.
\end{remark}

\subsection{The open stratum in $X_w$}

Any $\delta\in X_*(T)$ gives rise to a parabolic subgroup of $G$, 
$$
P(\delta) = \{g\in G: \lim_{t\to 0} \delta(t)g\delta(t)^{-1} \text{ exists in $G$}\},
$$
as well as an associated Levi subgroup,
$$
M(\delta) = \{g\in G: \lim_{t\to 0} \delta(t)g\delta(t)^{-1} = g\} = \{g\in G:\forall t,  \delta(t)g\delta(t)^{-1} = g \}.
$$

\begin{lemma}\label{veryeasy}
Suppose $Y$ is any irreducible (not necessarily closed) subvariety of $X_w$. Then there exists a $q\in W$ such that 
$$
Y\cap P(\delta)qB/B\cap X_w
$$
is dense in $Y$. 
\end{lemma}

\begin{proof}
For each $q\in W_\delta\setminus W$, set $Y_q:=Y\cap P(\delta)qB/B\cap X_w$ and let $\overline{Y_q}$ denote the closure of $Y_q$ (in $Y$). Clearly $Y = \bigcup Y_q$; hence also 
$
Y = \bigcup \overline{Y_q}.
$
By irreducibility, there is some $q$ such that $Y = \overline{Y_q}$. 
\end{proof}

\begin{lemma}\label{Schubert-translate}
Suppose $P$ is a standard parabolic subgroup of $G$, and let $v\in W$. Set $P':=vPv^{-1}$ and let $q$ and $w$ be Weyl group elements such that $P'qB/B\cap X_w$ is dense in $X_w$. Then there exists $t\in W$ such that $v^{-1}X_w\subseteq X_t$; moreover $t = v^{-1}rq$ for some $r\in vW_{P}v^{-1}$. 
\end{lemma}

\begin{proof}
Recall that 
$$
B w P = \bigcup_{v\in W_P} BwvB
$$
for any $w\in W$. So 
$$
B q^{-1} v P = \bigcup_{\hat r \in W_P} B q^{-1} v \hat r^{-1}B = \bigcup_{r \in vW_Pv^{-1}} B q^{-1}r^{-1}v B,
$$
where we reindex according to $r = v \hat r v^{-1}$ for the last equality. Therefore 
$$
vPv^{-1}qB/B = \bigcup_{r\in vW_Pv^{-1}} vBv^{-1}rqB/B.
$$

Since $vPv^{-1}qB/B\cap X_w$ is dense inside of $X_w$, 
there must exist some $r\in vW_Pv^{-1}$ such that $vBv^{-1}rqB/B\cap X_w$ is dense in $X_w$. Now observe that 
$$
X_w = \overline{vBv^{-1}rqB/B\cap X_w} \subseteq \overline{vBv^{-1}rqB/B} = vX_{v^{-1}rq}.
$$
With $t = v^{-1}rq$, the result follows. 
\end{proof}

Now we are ready for the main result of this section. 

\begin{theorem}
Suppose $\mu \in X^*(T)$ and $\lambda \in X^*(T)^+$ are such that $\mathbb{L}:=\C_\mu\otimes L_w(\lambda)$ has no semistable points in $X_w$ with respect to $T$. That is, 
$$
H^0(X_w,\mathbb{L}^{\otimes n})^T = 0
$$
for all $n\ge 1$. Then there is a fundamental coweight $x_i$ and a pair of Weyl group elements $u,v\in W$ such that 
\begin{enumerate}
\item $\langle \mu,vx_i\rangle  < \langle \lambda, ux_i\rangle$
\item $v^{-1}X_w\subseteq X_{u^{-1}}$.
\end{enumerate}
\end{theorem}

\begin{proof}
By Corollary \ref{open-dense-stratum}, we can find some $\delta\in X_*(T)$ and $d\in \R_{>0}$ such that $S_{d,\delta}^{\mathbb{L}}$ is open inside $X_w$. In particular, there is a dense subset $U\subset X_w$ such that 
$$
x\in U\implies \mu^{\mathbb{L}}(x,\delta)>0.
$$
By Lemma \ref{veryeasy}, let $q\in W$ be such that $U\cap P(\delta)qB/B\cap X_w$ is dense in $U$. Suppose $x\in \Omega_q$. By definition, $-\mu^{\mathbb{L}}(x,\delta)$ is the exponent of the $\delta$-induced $\C^*$-action on the fibre of $\mathbb{L}$ over $y:=\lim_{t\to 0} \delta(t)x$. Since $y$ belongs to $M(\delta)qB/B$, we readily calculate that 
$$
\mu^{\mathbb{L}}(x,\delta) = \langle \lambda,q^{-1}\delta\rangle-\langle \mu,\delta\rangle.
$$
Therefore $\langle \mu,\delta\rangle< \langle \lambda,q^{-1}\delta\rangle$. 
Now, we may find a $v\in W$ and $\delta_0:T\to \C^*$ such that $\delta = v \delta_0$ and $\delta_0$ is a dominant OPS with respect to $B$. Since $U\cap P(\delta)B/B\cap X_w$ is dense in $U$, $P(\delta)B/B\cap X_w$ is dense in $X_w$. By Lemma \ref{Schubert-translate}, we get
$$
v^{-1}X_w\subseteq X_t,
$$
where $t = v^{-1}rq$ for some $r\in vW_{P(\delta_0)}v^{-1}$ (i.e., $r$ fixes $\delta$). Take $u:=t^{-1}$; then (2) is satisfied. 

Moreover, $q^{-1}\delta = q^{-1}r^{-1}\delta = q^{-1}r^{-1}v \delta_0 = u \delta_0$. Therefore 
$$
\langle \mu,v \delta_0\rangle< \langle \lambda,u \delta_0\rangle.
$$
In particular, since $\delta_0$ is a nonnegative linear combination of fundamental coweights $x_i$, there exists some $x_i$ such that 
$$
\langle \mu,vx_i\rangle < \langle \lambda, ux_i\rangle.
$$
So the triple $(x_i, u,v)$ satisfy (1) as well. 
\end{proof}

Combining the theorem with Lemma \ref{borel-weil}, we obtain the following. 

\begin{corollary}\label{outside}
Suppose $\mu\not\in P_\lambda^w$. Then there exist a fundamental coweight $x_i$ and a pair of Weyl group elements $u,v\in W$ such that 
\begin{enumerate}
\item $\langle \mu,vx_i\rangle < \langle \lambda,ux_i\rangle$
\item $v^{-1}X_w\subseteq X_{u^{-1}}$.
\end{enumerate}
\end{corollary}

\section{Inequalities defining $P_\lambda^w$}\label{ineq-comb}

Corollary \ref{outside} gives sufficient conditions for $\mu\in P_\lambda^w$. 
In this section, we will give a combinatorial reinterpretation of criterion (2) above, as well as exhibit a smaller set of inequalities of the form $\langle \mu,vx_i\rangle\le \langle \lambda,ux_i\rangle$ that are necessary and sufficient for $\mu\in P_\lambda^w$. 

First, we find it helpful to recast the geometric criterion $v^{-1}X_w\subseteq X_{u^{-1}}$ above in terms of a combinatorial operation: namely, the Demazure product on $W$. After we have recalled some basic properties of the Demazure product, we can show that the necessity of the inequalities of (1) above follows simply from the Bruhat order (Lemma \ref{bound}). 

\begin{definition}
We say that a pair of Weyl group elements $u,v\in W$ is 
``$w$-trapping''
if $v^{-1}X_w\subseteq X_{u^{-1}}$. 
\end{definition}

\subsection{Properties of the Demazure product}

Given a (not necessarily reduced) word in simple reflections $(s_{i_1}, \hdots, s_{i_k})$ one can define their Demazure product inductively by setting
$$
s_{i_1}* \hdots * s_{i_k} := s_{i_1} *( s_{i_2}* \hdots * s_{i_k})
$$
and using the rule 
$$
s_i*w := \max\{w,s_iw\}.
$$

One can check that the alternate rule $w*s_i:=\max\{w,ws_i\}$ could have been used instead, with no change. Akin to the Demazure operators $D_i$, $*$ satisfies the braid relations and $s_i*s_i = s_i$ for every $i$. 

Then one defines the product $v*w$ of arbitrary Weyl group elements by first finding (any) reduced words $(s_{i_1},\hdots,s_{i_k})$ and $(s_{j_1},\hdots,s_{j_\ell})$ for $v$ and $w$ and setting 
$$
v*w:= s_{i_1}*\hdots*s_{i_k}*s_{j_1}*\hdots*s_{j_\ell}.
$$
This clearly makes $*$ into an associative product on $W$ with identity $e$. Note also that $w_0*v = v*w_0 = w_0$ for any $v$. 

\begin{proposition}
The Demazure product $s_{i_1}*\hdots*s_{i_k}$ is equal to the maximum product $s_{i_{j_1}}\cdots s_{i_{j_\ell}}$ over all subwords. 
\end{proposition}

\begin{proof}
If $k=0,1,2$ this is trivial. 

So assume $k\ge 3$. By induction we have that $w=s_{i_2}*\hdots *s_{i_k}$ is equal to $s_{i_{j_1}}\cdots s_{i_{j_\ell}}=:s_J$ for some subset $J = \{j_1,\hdots,j_\ell\}\subseteq [2,k]$. Moreover $w\ge s_{J'}$ for any subset $J'\subseteq [2,k]$. 

If $s_{i_1}w\ge w$, then $s_{i_1}*w = s_{i_1}w = s_{1\cup J}$. We of course have $s_{i_1}w \ge s_{J'}$ for all $J'\subseteq [2,k]$. But by \cite{BGG}*{Lemma 2.5(a)}, we also have $s_{1\cup J'}\le \max\{s_{i_1}w,w\}=s_{i_1}w$ for each $J'\subseteq [2,k]$. 

Otherwise, $s_{i_1}*w = w = s_{J}$. Still $w\ge s_{J'}$ for every $J'\subseteq[2,k]$. And once again by \cite{BGG}*{Lemma 2.5(a)}, each $s_{1\cup J'}\le \max\{s_{i_1}w,w\} = w$ for $J'\subseteq [2,k]$. 
\end{proof}

By the subword property of the Bruhat order, we get the following characterization. 

\begin{corollary}\label{biproduct}
The Demazure product $v*w$ is equal to $\op{max}\{xq: x\le v, q\le w\}$. 
\end{corollary}

In fact, this maximum can always be obtained as follows. 

\begin{proposition}\label{toom}
$$
v*w = \max\{xw : x\le v\}.
$$
\end{proposition}

\begin{proof}
If $v=e$, this is trivial. Otherwise, 
find some $s_jv \to v$. By inductive hypothesis, we may assume $(s_jv)*w = xw$ for some $x\le s_jv$. Now we know that 
$$
s_j*(s_jv) = v;
$$
hence 
$$
v*w = s_j*(s_jv)*w = s_j*(xw).
$$
The latter is equal to either $s_jxw$ or $xw$, whichever is greater. But we know that both $x\le v$ and $s_jx\le v$, so in either case we have $v*w = x'w$ for some $x'\le v$. 
\end{proof}

\subsection{Necessary conditions}

For any parabolic $P$, we use $W^P$ to denote the set of minimal-length coset representatives for $W/W_P$. 

We are now in a position to derive a practical set of necessary and sufficient conditions for when $\mu\in P_\lambda^w$. 
We begin by recasting the $w$-trapping pairs in a more useable framework. 

\begin{proposition}\label{eqeq}
The following are equivalent: 
\begin{enumerate}[label=(\alph*)]
\item The pair $(u,v)$ is $w$-trapping: $v^{-1}X_w \subseteq X_{u^{-1}}$. 
\item For every $q\le w$, $q^{-1} v \le u$. 
\item $w^{-1} * v \le u$. 
\end{enumerate}
\end{proposition}

\begin{proof}
(a) $\implies$ (b) follows from examining the $T$-fixed points of $X_w$, and the fact that taking inverses preserves the Bruhat order. ~\\

(b) $\implies$ (c) 
By Proposition \ref{toom}, $w^{-1}*v = q^{-1}v$ for some $q\le w$. Therefore $w^{-1}*v \le u$. ~\\

(c) $\implies$ (a) 
By induction on $\ell(v)$, we will show that $v^{-1}X_w\subseteq X_{u^{-1}}$ for $u = w^{-1}*v$, the base case $v = e$ (length $0$) being obvious. This clearly will imply the original statement.  

Assuming the implication holds for some $v$, we will show it for $v s_j$ such that $v\to vs_j$. Note that $ w^{-1}*(vs_j) = w^{-1}*v*s_j $ by definition. Set $u:= w^{-1}*v*s_j$ and $\hat u:= w^{-1}*v$. So $\hat u = u$ if $\hat us_j\to \hat u$ and $\hat u\to u= \hat us_j$ otherwise. 

Either way, we claim that $s_jX_{\hat u^{-1}}\subseteq X_{u^{-1}}$. Indeed, in the first case we have $s_j C_{\hat u^{-1}}\subseteq C_{\hat u^{-1}}\cup C_{s_j \hat u^{-1}}\subseteq X_{\hat u^{-1}} = X_{u^{-1}}$ (where the first inclusion follows by the Tits conditions on Bruhat order, see e.g. \cite{Springer}*{Lemma 8.3.7}). Hence $s_j X_{\hat u^{-1}}\subseteq X_{u^{-1}}$ since $X_{u^{-1}}$ is closed. 

In the second case, $s_j C_{\hat u^{-1}} = C_{u^{-1}}\subseteq X_{u^{-1}}$ (once again, by \cite{Springer}*{Lemma 8.3.7}) so $s_j X_{\hat u^{-1}}\subseteq X_{u^{-1}}$ again. 

The induction hypothesis gives that $v^{-1}X_w\subseteq X_{\hat u^{-1}}$, so 
$$
s_jv^{-1} X_w \subseteq s_j X_{\hat u^{-1}} \subseteq X_{u^{-1}}
$$
as desired. 
\end{proof}

\begin{remark}
Note that for each $v\in W$, there is a $w$-trapping pair $(u,v)$ with $u:=w^{-1}*v$. 

Since $u\ge w^{-1}*v\iff u^{-1} \ge v^{-1}*w$, we also get an equivalence
$$
v^{-1}X_w \subseteq X_{u^{-1}} \iff w^{-1}X_v \subseteq X_u.
$$
\end{remark}

In the following proposition, and throughout the remainder of this paper, we will use the following property relating the Bruhat order to inequalities on pairings. Specifically, whenever $v\le u$, we get a reverse relationship $u^{-1}\lambda \preceq v^{-1}\lambda$ for any dominant weight $\lambda$. Pairing with any fundamental coweight produces inequalities:

\begin{lemma}\label{bound}
Suppose $x$ is a dominant coweight, and that $u,v\in W$ satisfy $v \le u$. Then for any dominant weight $\lambda$, 
$$
\langle \lambda,ux\rangle \le \langle \lambda,vx\rangle.
$$
\end{lemma}

Combining Proposition \ref{eqeq} and Lemma \ref{bound}, we get necessary inequalities on $P_\lambda^w$. 

\begin{proposition}\label{inside}
Suppose $\mu\in P_\lambda^w$. Then for every maximal parabolic $P = P_i$ and $w$-trapping pair $(u,v)$, the inequality 
$$
\langle \lambda,u x_i\rangle\le \langle \mu,v x_i\rangle
$$
holds. 
\end{proposition}

\begin{proof}
By Theorem \ref{convexequal}, it suffices to verify the statement for each vertex $\mu = q\lambda$, where $q\le w$. Since $(u,v)$ is $w$-trapping, $q^{-1}v \le u$, so 
$$
\langle \lambda,u x_i\rangle \le \langle q \lambda, v x_i\rangle
$$
by Lemma \ref{bound}. 
\end{proof}

\subsection{Sufficient conditions}

Combining  Proposition \ref{inside}, Corollary \ref{outside}, and Proposition \ref{eqeq}, we have the following result defining the polytope $P_\lambda^w$ by linear inequalities.

\begin{theorem}\label{necsufineqs}
Let $\lambda$ be a dominant weight and $w\in W$. Then $\mu\in P_\lambda^w$ if and only if, for every maximal parabolic $P=P_i$ and $v\in W^P$, the inequality
\setcounter{theorem}{\value{theorem}+1}
\begin{align}\label{face-inequality}
\langle \lambda,ux_i\rangle\le \langle \mu,vx_i\rangle
\end{align}
holds, where 
$u = w^{-1}*v$. 
\end{theorem}

\begin{proof}
By Proposition \ref{inside}, the inequalities $\langle \lambda,ux_i\rangle \le \langle \mu,vx_i\rangle$ with 
$u = w^{-1}*v$ are necessary (indeed, $(u,v)$ is $w$-trapping). 

To show they are sufficient, consider $\mu\not \in P_\lambda^w$. Then by Corollary \ref{outside}, one can find $x_i, u, v$ such that $\langle \mu,vx_i\rangle < \langle \lambda,u x_i\rangle$; moreover these satisfy $u \ge w^{-1}*v$ by Proposition \ref{eqeq}. Therefore by Lemma \ref{bound}, $\langle \lambda,u x_i\rangle \le \langle \lambda,w^{-1}*v x_i\rangle$, so $\langle \mu,vx_i\rangle < \langle \lambda,w^{-1}*v x_i\rangle$. We can exchange $v$ for $\overline{v}$, the min-length representative in $vW_{P_i}$, since $\langle \mu,vx_i\rangle = \langle \mu,\bar v x_i\rangle$ and $w^{-1}*v \ge w^{-1}*\bar v$, by the subword characterization in Corollary \ref{biproduct}. 
\end{proof}

\subsection{Description of $P_{\lambda}^w$ as pseudo-Weyl polytopes}\label{pseudo}
We briefly give descriptions of Demazure polytopes $P_{\lambda}^w$ as pseudo-Weyl polytopes, following suggestions from Kamnitzer.

Recall from \cite{Kamn} that pseudo-Weyl polytopes admit three descriptions. They can be described as the convex hull $P(M(\mu_{\bullet}))=\op{conv}(\mu_{\bullet})$ where $\mu_{\bullet}=\{\mu_w| w \in W, \mu_{w} \leq_{w} \mu_u \forall w, u \in W\}.$
They can also be described as an intersection of cones: $P(M(\mu_{\bullet}))= \bigcap_{w \in W} C_w^{\mu_w}$, where we define 
$$
C_w^\mu = \{\nu\in X^*_\Q: \langle \nu, w x_i \rangle \ge \langle \mu, w x_i\rangle\text{ for all $i$}\}.
$$
Lastly they can be described as an intersection of half-spaces. Let $\Gamma=\{w x_i| w \in W, i \in 1, \dots rk(G)\}$, and let $M_{\bullet}$ be a collection of integers parametrized by $\Gamma$. Then $P(M_{\bullet})= \{\nu \in \mathfrak{h}^* | \langle \nu, \gamma \rangle \geq M_{\gamma} \forall \gamma \in \Gamma\}.$

This last description is easiest in light of what has already been shown: by Theorem \ref{necsufineqs} we have that $P_{\lambda}^w= P(M_{\bullet})$ where $M_{vx_i}= \langle \lambda, (w^{-1} *v) x_i \rangle$. 

In order to provide the ``vertex" or ``cone" descriptions, we introduce a map \[g: W \rightarrow \{q \in W | q \leq w\}\] \[v \mapsto v(v^{-1}*w).\]

\begin{lemma}
The map $g$ is well-defined, and if $v\le w$, then $g(vw_0) = v$. 
\end{lemma}
\begin{proof}
The map is well-defined due to Proposition \ref{toom}. Now suppose $v\le w$. Thus $vw_0 \ge w w_0$. From Proposition \ref{toom}, $(vw_0)^{-1}*w  = \max\{x^{-1}w | x \le vw_0\}$. This maximum is attained at $x = ww_0$, for there $x^{-1}w = w_0$. So 
\begin{align*}
g(vw_0) = vw_0((vw_0)^{-1}*w) = vw_0(w_0) = v. & \qedhere
\end{align*}
\end{proof}
This provides the ``vertex" description of the pseudo-Weyl polytope; namely, $$P_\lambda^w=\op{conv}\{g(v)\}_{v \in W}=\op{conv}\{v\lambda: v \leq w\}.$$

Finally, we formulate the ``cone" description of $P_\lambda^w$. 
\begin{proposition}
The polytope $P_\lambda^w$ can be described as 
$$
P_\lambda^w = \bigcap_{v \in W} C_v^{\mu_v},
$$
where $\mu_v = g(v)\lambda$. 
Moreover, for arbitrary $v_1,v_2$, 
$$
v_1^{-1}\mu_{v_2}\ge v_1^{-1} \mu_{v_1}.
$$
\end{proposition}

\begin{proof}
We have 
\begin{align*}
P_\lambda^w &= \{\mu | \langle \mu,vx_i\rangle \ge \langle \lambda, w^{-1}*v x_i\rangle, \forall v\in W, 1\le i\le r\}\\
&=\bigcap_{v\in W} \{\mu | \langle \mu,vx_i\rangle \ge \langle \lambda, w^{-1}*v x_i\rangle,1\le i\le r\}\\
&=\bigcap_{v\in W} \{\mu | \langle \mu,vx_i\rangle \ge \langle v(v^{-1}*w)\lambda, v x_i\rangle,1\le i\le r\}\\
&= \bigcap_{v\in W} C_v^{\mu_v}.
\end{align*}

Now let $v_1,v_2$ be arbitrary. By Proposition \ref{toom}, write $v_i^{-1}*w = v_i^{-1}q_i$ for suitable $q_i\le w$. Thus $\mu_{v_i} = g(v_i)\lambda = q_i\lambda$. Now $v_1^{-1}q_2$ belongs to the set $\{v_1^{-1}q | q\le w\}$, which has maximum $v_1^{-1}*w = v_1^{-1}q_1$. Therefore $v_1^{-1}q_2\le v_1^{-1}q_1$. This directly implies that $v_1^{-1}q_2\lambda \ge v_1^{-1}q_1 \lambda$.
\end{proof}

We again refer to \cite{Dykes} for additional work on the combinatorics of Demazure polytopes.

\section{Demazure structure on the faces} \label{faces}

Now that we know the inequalities determining $P_\lambda^w$, for an arbitrary inequality we can examine the locus of $P_\lambda^w$ where that inequality is sharp. Such a locus is called a face. 

\begin{definition}
Suppose that $P_i$ is a maximal parabolic and $v\in W^P$. The collection of points in $P_\lambda^w$ satisfying (\ref{face-inequality}) with equality is a face of $P_\lambda^w$, denoted 
$$
\mathcal{F}(v,P_i) = \{\mu\in P_\lambda^w: \langle \lambda,(w^{-1}*v) x_i\rangle = \langle \mu,vx_i\rangle\}.
$$
\end{definition}

In this section, we will determine the set of vertices on an arbitrary face $\mathcal{F}(v,P_i)$. In the process we uncover a robust, underlying structure: namely, the face is itself a Demazure polytope for a Levi subgroup.

For the remainder of the section, we fix a maximal parabolic $P=P_i$ and coset representative $v\in W^P$. The function $W\to W^P$, which assigns to each $w$ the minimal-length representative for $wW_P$, we will denote $w\mapsto \bar w$, the dependence upon $P$ being understood. 

For a dominant weight $\lambda$, let $W^\lambda$ denote the set of minimal-length coset representatives in $W/W_\lambda$, where $W_\lambda$ is the subgroup of $W$ stabilizing $\lambda$. 
Note that $V_\lambda^w = V_\lambda^{w'}$ whenever $wW_\lambda = w'W_\lambda$, so we can and generally will assume that $w = w_1 w_0^\lambda$, where $w_1\in W^\lambda$ and $w_0^\lambda$ is the longest element in $W_\lambda$ (so $w$ is a \emph{maximal}-length coset representative.) This is merely for convenience in stating Proposition \ref{threading} below. 

Suppose we wish to find exactly which vertices $x\lambda$, for $x\in W$, $x\le w$, lie on the face $\mathcal{F}(v,P_i)$. The requirement is that 
$$
 \langle \lambda,u x_i\rangle  = \langle x\lambda, vx_i \rangle,
$$
where $u = \overline{w^{-1}*v}$. 
By Corollary \ref{biproduct}, $x\le w$ implies $u\ge \overline{x^{-1}v}$.

\begin{lemma}
Fix a dominant weight $\lambda$, dominant coweight $\delta$, and $v\in W$. 
Suppose for some $u\ge v$ that $\langle \lambda,u\delta\rangle  = \langle \lambda,v\delta\rangle $.  Then 
$$
W_\lambda u W_\delta = W_\lambda v W_\delta.
$$
\end{lemma}

\begin{proof}
We induct on $\ell(u)-\ell(v)$, the base case $0$ being trivial. 

Find positive roots $\gamma_i$ such that 
$$
v \to s_{\gamma_1} v  \to  s_{\gamma_2} s_{\gamma_1} v \to \cdots \to s_{\gamma_k} \cdots s_{\gamma_1} v = u.
$$
Set $\hat v = s_{\gamma_{k-1}}\cdots s_{\gamma_1} v = s_{\gamma_k}u$. Since $v\le \hat v\le u$, we certainly have $\langle \lambda,v \delta\rangle \ge \langle \lambda,\hat v\delta\rangle \ge \langle \lambda,u\delta\rangle$; therefore $\langle \lambda,\hat v\delta\rangle = \langle \lambda,u\delta\rangle$ and $\langle \lambda,v\delta\rangle = \langle \lambda,\hat v\delta\rangle$. 

Now $u\delta = \hat v\delta - \langle \hat v^{-1}\gamma_k,\delta\rangle \gamma_k^\vee$, and $\hat v^{-1}\gamma_k$ is a positive root. Two cases can occur: 
\begin{enumerate}
\item If $\langle \hat v^{-1}\gamma_k,\delta\rangle = 0$, set $\eta=\hat v^{-1}\gamma_k$. Then $s_\eta \in W_\delta$ and $u = s_{\gamma_k} \hat v = \hat v (\hat v^{-1} s_{\gamma_k} \hat v) = \hat vs_\eta$. 
\item If $\langle \hat v^{-1}\gamma_k,\delta\rangle\ne 0$, then it must be that $\langle \lambda,\gamma_k^\vee\rangle = 0$, so $s_{\gamma_k}\in W_\lambda$. 
\end{enumerate}
Either way, $W_\lambda u W_\delta = W_\lambda \hat v W_\delta$. 

By inductive hypothesis, $W_\lambda \hat v W_\delta = W_\lambda v W_\delta$, which concludes the proof. 
\end{proof}

This leads us to the following. 

\begin{proposition}\label{threading}
The vertices of $P_\lambda^w$ lying on the face $\mathcal{F}(v,P_i)$ form the set 
$$
\{x\lambda: x \in vW_{P_i}u^{-1}, x\le w \}.
$$
\end{proposition}

\begin{proof}
Clearly if $x = vcu^{-1}$ for some $c\in W_{P_i}$, then 
$\langle x\lambda,vx_i\rangle = \langle \lambda,ux_i\rangle$. 

On the other hand, assuming $\langle x\lambda,vx_i\rangle = \langle \lambda,ux_i\rangle $ and $x\le w$, the lemma says there exist $b\in W_\lambda$ and $c\in W_{P_i}$ such that 
$$
bx^{-1}vc = u;
$$
so $xb^{-1} = vcu^{-1}$. Now, with the assumption that $w$ is a maximal-length coset representative, $x\le w \iff x x' \le w$ for any $x'\in W_\lambda$. Therefore the vertices $x\lambda$ and $(xb^{-1})\lambda$ coincide, and
$xb^{-1}\in vW_{P_i}u^{-1}, xb^{-1}\le w$. 
\end{proof}

Another way to say this is that $x\lambda$ lies on the face if and only if (up to changing coset representative for $xW_\lambda$) $xu\in v\cdot [e,w_0^P]\cap [e,w]\cdot u$, an intersection of translated Bruhat intervals.

Note that when $w = w_0$, any $v$ gives $w^{-1}*v = w_0$. So the set of $x$'s is just $vW_{P_i}w_0$. 
In general, we may find a $y\in W_{P_i}$ such that 
\setcounter{theorem}{\value{theorem}+1}
\begin{align}\label{what-is-y}
uy^{-1} = w^{-1}*v
\end{align} (and moreover $\ell(uy^{-1}) = \ell(u)+\ell(y^{-1})$).

\begin{proposition}\label{ival}
The set $W_{P_i}\cap v^{-1}[e,w]u$ is a Bruhat interval inside $W_{P_i}$, namely $[e,y]$. 
\end{proposition}

Before we come to the proof, we formulate two technical lemmas. 
First, if $u,v$ satisfy $\bar u = \bar v$, then they are related in the Bruhat order if and only if they are related by reflections from $\Phi_P$: 
\begin{lemma}\label{LeviBruhat}
Suppose $uW_{P_i} = vW_{P_i}$ and $u\le v$. Then there exists a sequence of roots $\eta_i\in \Phi_P^+$ such that 
$$
u = v s_{\eta_1}\cdots s_{\eta_m} < \hdots < v s_{\eta_1} < v.
$$
\end{lemma}

\begin{proof}
We induct on $\ell(v)-\ell(u)$, the case $u=v$ being the trivial base case. 

Supposing $u< v$, find any positive root $\beta$ such that $u \le v s_\beta < v$. Then 
$$
\bar u \le \overline{ v s_{\beta}} \le \bar v,
$$
so in particular $vs_{\beta} W_{P_i} = v W_{P_i}$. Hence $s_\beta W_{P_i} = W_{P_i}$, so $s_\beta\in W_{P_i}$, so $\beta\in \Phi_P^+$. The inductive hypothesis applies to the pair $u\le v s_{\beta}$, which gives the result. 
\end{proof}

Second, the Bruhat order on $W_{P_i}$ is preserved by suitable left- and right-multiplication by elements of $W^P$. 

\begin{lemma}\label{b1}
Suppose $u,v\in W^P$, $y\in W_{P_i}$. Let $\gamma\in \Phi_P^+$, and set $\beta:=v\gamma\in \Phi^+$. Then 
\begin{align*}
s_\gamma y < y &\iff s_{\beta} v y u^{-1}  < v y u^{-1}\\
&\iff s_{\beta} q  < q,
\end{align*}
where $q=v y u^{-1}$.
\end{lemma}

\begin{proof}
We know that $s_\gamma y < y$ if and only if $y^{-1}\gamma \prec 0$, which is if and only if $u y^{-1} \gamma \prec 0$ since $y^{-1}\gamma\in \Phi_P$ and $u\in W^P$. But of course $u y^{-1}\gamma = u y^{-1} v^{-1} \beta$, so this last condition is equivalent to $s_\beta v y u^{-1} < v y u^{-1}$. 
\end{proof}

Note that the lemma applies in the special cases $u=e$ or $v=e$. 

\begin{proof}[Proof of Proposition \ref{ival}]
Set $\Omega = W_{P_i}\cap v^{-1}[e,w]u$. Recall $y= (v^{-1}*w)u$ (equation (\ref{what-is-y})). 

By Lemma \ref{toom}, $w^{-1}*v$ is equal to some $q^{-1}v$ where $q\le w$. Therefore $u = q^{-1}vy$ and 
$$
y = v^{-1}\cdot q \cdot u \in \Omega.
$$

Now let $y'\le y$, and find a sequence 
$$
y' = s_{\gamma_k} \cdots s_{\gamma_1} y < \hdots < s_{\gamma_1} y < y.
$$
There is a corresponding sequence $\beta_i:=v\gamma_i$, each a positive root since $v\in W^P$. By Lemma \ref{b1}, applied repeatedly, we obtain 
$$
q':= s_{\beta_k} \cdots s_{\beta_1} q < \hdots < s_{\beta_1} q < q
$$
(recall that $q = vy u^{-1}$). This implies that $q'\le w$. So we have 
\begin{align*}
y' = v^{-1}v y' &= v^{-1} v s_{\gamma_k}\cdots s_{\gamma_1} y =v^{-1} s_{\beta_k}\cdots s_{\beta_1} v y \\ &= v^{-1} s_{\beta_k}\cdots s_{\beta_1} qu = v^{-1}q'u \in \Omega.
\end{align*}

Next, start with any $q' \le w$ such that $y':=v^{-1} q' u \in \Omega$ (i.e., $y'\in W_{P_i}$). We will show that $y'\le y$. Indeed, $q'^{-1} v \le q^{-1} v$ by Lemma \ref{toom}, which means $u y'^{-1} \le u y^{-1}$. By Lemma \ref{LeviBruhat}, there exists a sequence of roots $\gamma_i\in \Phi_P^+$ such that 
$$
uy'^{-1} = uy^{-1} s_{\gamma_1} \cdots s_{\gamma_m} <\hdots <uy^{-1}s_{\gamma_1} <uy^{-1}. 
$$
Taking inverses, 
$$
y'u^{-1} < \hdots < s_{\gamma_1} y u^{-1} < y u^{-1}.
$$
Applying Lemma \ref{b1} (with $v=e$), 
$$
y' < \hdots < s_{\gamma_1} y < y,
$$
as desired. \end{proof}

As a corollary, we find that faces of $P_\lambda^w$ are Demazure polytopes associated to Levi subgroups; see Figure \ref{peel} for a depiction of a Demazure polytope, highlighting some of its faces. 
We use $P_{L,\mu}^y$ to denote a Demazure polytope for $L$, to distinguish it from the notation for $G$. 

\begin{corollary}\label{facetDemazurestructure1}
Let $\mathcal{F}=\mathcal{F}(v,P_i)$ be a face of $P_\lambda^w$. Let $y = (w^{-1}*v)^{-1}u$ as above. The map 
\begin{align*}
\mathcal{F} &\to P_{L,u^{-1}\lambda}^y \\
\mu &\mapsto v^{-1}\mu
\end{align*}
is well-defined and a linear isomorphism of convex polytopes. 
\end{corollary}


\begin{remark} \label{BIP}
Let $G=A_n$, and $[u, v] \subseteq W=S_{n+1}$ be an interval in the symmetric group. In \cite{TW}, Tsukerman and Williams studied the combinatorial aspects of the \emph{Bruhat interval polytope}
$$
Q_{u,v}:=\op{conv}\{z(\rho): z \in [u,v]\},
$$
where $\rho$ is the half sum of positive roots. In particular, using techniques from matroid theory, they derived inequalities describing the polytopes $Q_{u,v}$ and showed that the faces of Bruhat interval polytopes are again Bruhat interval polytopes. It would be interesting to know if these inequalities and face descriptions agree with those of Theorem \ref{necsufineqs} and Corollary \ref{facetDemazurestructure1} when considering the common polytopes $Q_{e,w}=P^w_{\rho}$ in this setting.

\end{remark}

\tdplotsetmaincoords{70}{115}
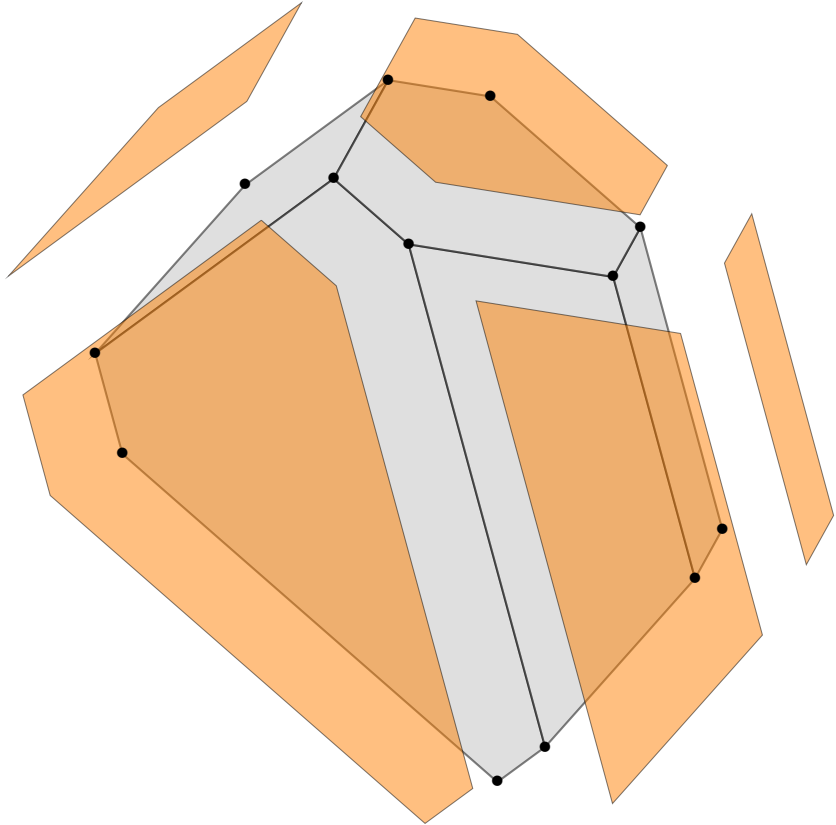
\begin{figure}
\begin{center}
\begin{tikzpicture}[tdplot_main_coords,scale=1.5]
   
\coordinate(p0) at (-3.0, 1.0, 0.707107);
\coordinate(p1) at (-3.0, -1.0, 0.707107);
\coordinate(p2) at (-2.5, 1.5, 1.414214);
\coordinate(p3) at (-2.5, -1.5, 1.414214);
\coordinate(p4) at (-1.5, 2.5, -1.414214);
\coordinate(p5) at (-1.5, 0.5, 2.828427);
\coordinate(p6) at (-1.5, -0.5, 2.828427);
\coordinate(p7) at (-1.0, 3.0, -0.707107);
\coordinate(p8) at (-0.5, 1.5, -2.828427);
\coordinate(p9) at (0.5, 1.5, -2.828427);
\coordinate(p10) at (1.0, 3.0, -0.707107);
\coordinate(p11) at (1.5, 2.5, -1.414214);
\coordinate(p12) at (1.5, 0.5, 2.828427);
\coordinate(p13) at (1.5, -0.5, 2.828427);
\coordinate(p14) at (2.5, 1.5, 1.414214);
\coordinate(p15) at (2.5, -1.5, 1.414214);
\coordinate(p16) at (3.0, 1.0, 0.707107);
\coordinate(p17) at (3.0, -1.0, 0.707107);
   
\draw[thick,fill=lightgray,opacity=0.5] (p3) -- (p6) -- (p13) -- (p15) -- (p3);
\draw[thick,fill=lightgray,opacity=0.5] (p3) -- (p1) -- (p8) -- (p9) -- (p17) -- (p15) -- (p3);
\draw[thick,fill=lightgray,opacity=0.5] (p0) -- (p1) -- (p8) -- (p4) -- (p0);
\draw[thick,fill=lightgray,opacity=0.5] (p0) -- (p2) -- (p7) -- (p4) -- (p0);
\draw[thick,fill=lightgray,opacity=0.5] (p2) -- (p0) -- (p1) -- (p3) -- (p6) -- (p5) -- (p2);

\draw[fill=orange,opacity=0.5] (-3.065685, 1.5, 1.814214) -- (-3.565685, 1.0, 1.107107) -- (-3.565685, -1.0, 1.107107) -- (-3.065685, -1.5, 1.814214) -- (-2.065685, -0.5, 3.228427) -- (-2.065685, 0.5, 3.228427) -- (-3.065685, 1.5, 1.814214) -- (-3.065685, 1.5, 1.814214);

\draw[fill=orange,opacity=0.5] (-3.742462, 1.742462, 0.707107) -- (-3.242462, 2.242462, 1.414214) -- (-1.742462, 3.742462, -0.707107) -- (-2.242462, 3.242462, -1.414214) -- (-3.742462, 1.742462, 0.707107);
   
\draw[fill=orange,opacity=0.5] (-2.5, -2.348528, 2.014214) -- (-1.5, -1.348528, 3.428427) -- (1.5, -1.348528, 3.428427) -- (2.5, -2.348528, 2.014214) -- (-2.5, -2.348528, 2.014214);

\draw[fill=orange,opacity=0.5] (-2.5, -2.207107, 0.914214) -- (-3.0, -1.707107, 0.207107) -- (-0.5, 0.792893, -3.328427) -- (0.5, 0.792893, -3.328427) -- (3.0, -1.707107, 0.207107) -- (2.5, -2.207107, 0.914214) -- (-2.5, -2.207107, 0.914214);

\draw[fill=orange,opacity=0.5] (-4.414214, 1.0, -0.292893) -- (-4.414214, -1.0, -0.292893) -- (-1.914214, 1.5, -3.828427) -- (-2.914214, 2.5, -2.414214) -- (-4.414214, 1.0, -0.292893);

\node at (p0) {\small $\bullet$};
\node at (p1) {\small $\bullet$};
\node at (p2) {\small $\bullet$};
\node at (p3) {\small $\bullet$};
\node at (p4) {\small $\bullet$};
\node at (p5) {\small $\bullet$};
\node at (p6) {\small $\bullet$};
\node at (p7) {\small $\bullet$};
\node at (p8) {\small $\bullet$};
\node at (p9) {\small $\bullet$};

\node at (p13) {\small $\bullet$};

\node at (p15) {\small $\bullet$};

\node at (p17) {\small $\bullet$};

\end{tikzpicture}
\caption{\label{peel} \small Faces of Demazure polytopes are Demazure polytopes: the original polytope has type $A_3$, so the facets have type $A_2$ or $A_1\times A_1$, which forces them to be hexagons, trapezoids, triangles, or rectangles.}
\end{center}
\end{figure}

If we wish to remove the ``twist'' (by $v^{-1}$) in the 
isomorphism of Corollary \ref{facetDemazurestructure1}, we certainly may, at the expense of losing the ``standardness'' of the Levi. In other words, 
$$
\mathcal{F} = P_{vLv^{-1},vu^{-1}\lambda}^{vyv^{-1}}.
$$
This Demazure polytope for $vLv^{-1}$ suggests that there may actually be a $vLv^{-1}$-Demazure \emph{module} sitting inside $V_\lambda^w$ (with weight polytope equal to $\mathcal{F}$). Indeed there is: set $B_L:=B\cap L$ and recall that $vB_Lv^{-1}\subseteq B$ (since $v\in W^P$), so $V_\lambda^w$ is canonically a $vB_Lv^{-1}$-module. 

\begin{proposition}\label{facetDemazurestructure2}
Every map in the following commutative diagram is a $vB_Lv^{-1}$-equivariant inclusion. Here $q = v(v^{-1}*w)$ and $uy^{-1} = w^{-1}*v$. 
\begin{center}
\begin{tikzcd}
V_{vLv^{-1},vu^{-1}\lambda}^{vyv^{-1}} \arrow[r] \arrow[d]& V_\lambda^q \arrow[r] & V_{\lambda}^w \arrow[d] \\
V_{vLv^{-1},vu^{-1}\lambda} \arrow[rr] & &  V_\lambda
\end{tikzcd}
\end{center}
\end{proposition}

\begin{proof}
One readily checks that $vu^{-1}\lambda$ is a highest weight for the (possibly reducible) $vLv^{-1}$-module $V_\lambda$. That is to say, $vB_Lv^{-1}$ stabilizes the line spanned by $vu^{-1} v_\lambda$; this is because $u\in W^P$ and $B$ stabilizes $\C v_\lambda$. Therefore there exists an $vLv^{-1}$-submodule of $V_\lambda$ isomorphic to $V_{vLv^{-1},vu^{-1}\lambda}$ with highest weight vector $vu^{-1}v_\lambda$. Moreover, in the decomposition of $V_\lambda$ into irreducible $vLv^{-1}$-modules, there is only one submodule with this highest weight, since $L$ contains the full torus and the weight space $V_\lambda(vu^{-1}\lambda)$ is $1$-dimensional. 

Recall that $vyv^{-1}vu^{-1} = vyu^{-1} = v(v^{-1}*w) = q$. 
By definition, $V_{vLv^{-1},vu^{-1}\lambda}^{vyv^{-1}}$ is the smallest $vB_Lv^{-1}$-submodule of $V_{vLv^{-1},vu^{-1}\lambda}$ containing $qv_\lambda$. It is therefore also the smallest $vB_Lv^{-1}$-submodule of $V_\lambda$ containing $qv_\lambda$, and is naturally contained in $V_\lambda^q$ -- the smallest $B$-submodule of $V_\lambda$ containing $qv_\lambda$. Since $q\le w$ (Lemma \ref{toom}), we also have $V_\lambda^q\subseteq V_\lambda^w$. 
\end{proof}

To summarize, Proposition \ref{facetDemazurestructure2} manifests a $vB_Lv^{-1}$-Demazure module (for $vLv^{-1}$) sitting inside $V_\lambda^w$. Its highest weight is $vu^{-1}\lambda$ and lowest weight is $q\lambda$; moreover $vyv^{-1}$ is the element of $vW_Lv^{-1}$ that takes the highest weight to the lowest weight, since $vyu^{-1}=q$. 
Together, Corollary \ref{facetDemazurestructure1} and Proposition \ref{facetDemazurestructure2} will be used to complete the induction step in proving Theorem \ref{lattice}.

\section{Lattice points in Demazure polytopes} \label{IntegralPoints}

Our main goal for the rest of this paper is to establish Theorem \ref{lattice}, that the characters $\op{char} V_\lambda^w$ are saturated. For each dominant weight $\lambda$, we will be interested in the shifted lattice $\lambda+Q$ and refer to its elements as ``lattice points,'' the reference to $\lambda$ being understood. 
Succinctly, we wish to show that if $\mu \in P_{\lambda}^w$ is a lattice point, then $V_{\lambda}^w(\mu) \neq 0$. 

First we recall a lemma that will be needed in the proof of the main theorem. 

\begin{lemma}\label{nls}
Suppose $\mu\in P_\lambda^w$ and $s_iw \to w$. Then $(\mu+\Q\alpha_i)\cap P_\lambda^{s_iw}$ is nonempty. 
\end{lemma}

\begin{proof}
This follows directly from Corollary \ref{maxs-are-equal}. 
\end{proof}
We now relate the saturation property to the line segments introduced earlier. We call $\mu + \mathbb{Q}\alpha_i$ the $\alpha_i$-line through $\mu$. If $(\mu+\mathbb{Q}\alpha_i) \cap P_\lambda^{s_{i_1}w} \neq 0$, then this intersection is a line segment (we do not require the endpoints to be distinct) with endpoints $\mu'$ and $\mu''$ as introduced in Section \ref{sectionconvexitysetup}. It is natural to attempt to prove Theorem \ref{lattice} using induction. In the next proposition, we isolate the difficult part of the inductive step, so that it may be analyzed separately.

We begin by fixing a reduced word $w=s_{i_1} \dots s_{i_k}$ so that $s_{i_1}w \rightarrow w$. Given a lattice point $\mu \in P^w_{\lambda}$, it is natural to try to understand $V_{\lambda}^w(\mu)$ in terms of some weight space $V_{\lambda}^{s_{i_2} \dots s_{i_k}}(\nu)$ for some nontrivial weight $\nu=\mu+n\alpha_{i_1}$. 
\begin{proposition}\label{indsketch}
	Let $\mu \in P^w_{\lambda}$ be a lattice point and $s_{i_1}w \rightarrow w$ as above. Assume that $\op{char} V_\lambda^{s_{i_2}\dots s_{i_k}}$ is saturated. Let $\mu''.\mu'$ denote the associated $\alpha_{i_1}$-line segment $(\mu+\mathbb{Q}\alpha_{i_1})\cap P_{\lambda}^{s_{i_2} \dots s_{i_k}}$ (which is not empty by Lemma \ref{nls}). If the line segment $\mu''.\mu'$ contains a lattice point $\nu \in \lambda+Q$, then $V_{\lambda}^w(\mu) \neq 0$ (note that we do not assert that $\mu', \mu''$ are lattice points). 
\end{proposition}

\begin{proof}	
	Since $P^{w}_{\lambda}$ is stable under the action of $s_{i_1}$, and since $\mu \in P_{\lambda}^w$, we must have both $\mu$ as well as $s_{i_1}\mu \in P^w_{\lambda}$. Thus we see that we must have both \[s_{i_1} \mu' \leq \mu \leq \mu'\] and similarly \[s_{i_1} \mu' \leq s_{i_1} \mu \leq \mu'\] in the dominance order on $\mu+\Q\alpha_{i_1}$. We write $\mu'_{int}$ for the largest lattice point in this intersection, namely $\mu'_{int} \in (\lambda+Q) \cap \mu''.\mu'$ and $\mu'_{int}+n \alpha_{i_1} \notin (\lambda+Q) \cap \mu''.\mu'$ for any $n \in \mathbb{Z}^{>0}$. Note that $\mu'_{int}$ exists since $\mu''.\mu' \cap (\lambda+Q)\ne \emptyset$ by assumption. Now since both $\mu$ and $ s_{i_1} \mu$ are in $\lambda+Q$, we actually have \[s_{i_1}\mu'_{int} \leq \mu \le \mu'_{int}\] and \[s_{i_1} \mu'_{int} \le s_{i_1}\mu \leq \mu'_{int}.\]
	
	Now consider the application of the Demazure operator $D_{\alpha_{i_1}}$ to $\op{char} V_{\lambda}^{s_{i_2} \dots s_{i_k}}$. By hypothesis, $e^{\mu'_{int}}$ appears in $\op{char} V_{\lambda}^{s_{i_2} \dots s_{i_k}}$ with positive coefficient. By Proposition \ref{DemProps}(1) and Corollary \ref{sl2-stable}, both $e^{\mu'_{int}}$ and $e^{s_{i_1} \mu'_{int}}$ appear in $\op{char} V_\lambda^w = D_{\alpha_{i_1}}\op{char} (V_{\lambda}^{s_{i_2} \dots s_{i_k}})$ with the same positive multiplicity. Moreover Corollary \ref{sl2-stable} guarantees that every lattice point on the $\alpha_{i_1}$-string beween $\mu'_{int}$ and $s_{i_1} \mu'_{int}$ must also appear with positive multiplicity. Thus we find that $e^{\mu}$ (and $e^{s_{i_1}\mu}$) appear with positive multiplicity in $D_{\alpha_{i_1}}\op{char}(V_{\lambda}^{s_{i_2} \dots s_{i_k}})$. In other words, $V_{\lambda}^w(\mu) \neq 0$.
	\end{proof}

In view of the above proposition, it is sufficient for us to find nontrivial lattice points on $\mu''.\mu'$. This is our next goal, whose proof will occupy the next two sections. For technical reasons which will become clear, we consider two cases: $\mathfrak{g}$ of type $A_r$ and $\mathfrak{g}$ of type $B_r,C_r,$ or $D_r$.

\section{Saturation in Type A}\label{sectiontypeA}
The situation in type $A_r$ is much simpler, in that it is merely a corollary of the inequalities defining $P_{\lambda}^w$. We come to the proof immediately. Let $\mu$ be a lattice point in $P_{\lambda}^w$ with $s_{i_1}w \rightarrow w$, and let $\mu''.\mu'=(\mu+\mathbb{Q}\alpha_{i_1}) \cap P_{\lambda}^{s_{i_2} \dots s_{i_k}}$. 

	\begin{proposition}\label{TypeAInt-2}
	In type $A_r$, the endpoints of the line segment $\mu''.\mu'$ are lattice points.
\end{proposition}

\begin{proof}
Since $\mu \in P_{\lambda}^w$, it satisfies our inequalities from Theorem \ref{necsufineqs}; in other words 
\[\langle \lambda, (w^{-1}*v) x_i \rangle \leq \langle \mu, v x_i \rangle \] 
for all maximal parabolics $P_i$ and for all $v \in W^{P_i}$. We also have that $\mu'=\mu+k\alpha_{i_1}$ must satisfy all inequalities defining $P_{\lambda}^{s_{i_2} \dots s_{i_k}}$.

We claim that $\mu'$ lies on some face $\mathcal{F}(v,P_i)$ of $P_\lambda^{s_{i_2}\dots s_{i_k}}$, such that 
$$
\langle \alpha_{i_1},vx_i\rangle \ne 0.
$$
Indeed, suppose that every face $\mathcal{F}(v,P_i)$ containing $\mu'$ satisfies $\langle \alpha_{i_1},vx_i\rangle = 0$. Then $\mu'+\epsilon \alpha_{i_1}$ will satisfy every inequality defining $P_{\lambda}^{s_{i_2} \dots s_{i_k}}$ if $\epsilon>0$ is small enough. This contradicts the maximality of $\mu'$.

We now choose $\mathcal{F}(v,P_i)$ to be such a face:  
\[\langle \lambda, \left((s_{i_2} \dots s_{i_k})^{-1}*v\right) x_i \rangle = \langle \mu+k\alpha_{i_1}, v x_i \rangle.\] 
This in turn becomes 
\[\langle \lambda, \left((s_{i_2} \dots s_{i_k})^{-1}*v\right) x_i\rangle -\langle \mu, vx_i\rangle = k \langle \alpha_{i_1}, vx_i \rangle.\] 

Since $\lambda+Q = \mu+Q$ and hence $u_1\lambda+Q = u_2 \mu+Q$ for any pair $u_1,u_2\in W$, 
the left hand side is an integer. In type $A_r$, the possible pairings $\langle \alpha_{i_1}, vx_i \rangle = \langle v^{-1} \alpha_{i_1}, x_i \rangle$ belong to $\{-1,0,1\}$ (no root in type $A_r$, when expressed as a sum of simple roots $\beta = \sum c_i \alpha_i$, can have any coefficients $c_i$ other than $-1, 0$ or $1$.) Since we picked $v$ such that $\langle \alpha_{i_1},vx_i\rangle\ne0$, 
 $k$ is an integer, and we have shown that $\mu'=\mu+k\alpha_{i_1}\in \lambda+Q$ is a lattice point.
 
 A similar argument holds for $\mu''$. 
\end{proof}

\begin{theorem}\label{A-result}
	In type $A_r$, Demazure characters $\op{char}(V_{\lambda}^{w})$ are saturated. 
\end{theorem}

\begin{proof}
	We proceed by induction on the length of $w$. For $\ell(w)=0$ this is trivial, and moreover it is clear for $\ell(w)=1$, since in this case $w=s_i$ for some simple reflection and the Demazure module $V_{\lambda}^w$ is precisely an irreducible representation of $SL_{2,\alpha_i}$ with highest weight $\lambda$ and lowest weight $s_{i}\lambda$. We now assume that the proposition holds for $\ell(w)=k-1$. Let $\ell(w)=k$ and fix a reduced word $w=s_{i_1} \dots s_{i_k}$. Given any lattice point $\mu \in P_{\lambda}^w$, consider the $\alpha_{i_1}$-string through $\mu$, and the associated line segment $\mu''.\mu'$, which exists by Lemma \ref{nls}. By Proposition \ref{TypeAInt-2}, there is some lattice point $\mu'_{int} \in \mu''.\mu'$. By our inductive assumption, $V_{\lambda}^{s_{i_2} \dots s_{i_k}}(\mu'_{int}) \neq 0$. We now apply Proposition \ref{indsketch} to see that $V_{\lambda}^w(\mu) \neq 0$.
	\end{proof}

\section{Saturation in types B,C,D}\label{sectiontypesbcd}
For this entire section we assume that $\mathfrak{g}$ is simple of type $B_r,C_r$ or $D_r$; our results on the line segments $\mu''.\mu'$ admit a uniform treatment in these remaining classical types. We begin with a more careful analysis of the inequalities defining $P_{\lambda}^w$. First note that we may use the exact same argument as in Section \ref{sectiontypeA} to try to understand weights along integral root strings which lie on a face of a polytope $P_{\lambda}^{s_{i_1}w}$. To be precise, for the remainder of this section, let $\mu$ be a lattice point in $P_{\lambda}^w$ with $s_{i_1}w \rightarrow w$, and let $\mu''.\mu'=(\mu+\mathbb{Q}\alpha_{i_1}) \cap P_{\lambda}^{s_{i_2} \dots s_{i_k}}$. 

\begin{lemma}\label{typebcdhalfint}
	Let $\mathfrak{g}$ be of type $B_r,C_r$ or $D_r$. Then the line segment $\mu''.\mu'$ must have half-integral endpoints, or in other words $\mu'',\mu' \in \lambda+\frac{1}{2}Q$.
\end{lemma} 

\begin{proof}
	As in the proof of Proposition \ref{TypeAInt-2}, since $\mu'$ is on a face of $P_{\lambda}^{s_{i_2} \dots s_{i_k}}$, at least one of these inequalities is an equality, giving   
	\[\langle \lambda, \left((s_{i_2} \dots s_{i_k})^{-1}*v\right) x_i \rangle = \langle \mu+k\alpha_{i_1}, v x_i \rangle.\] 
	This in turn becomes 
	\[\langle \lambda, \left((s_{i_2} \dots s_{i_k})^{-1}*v\right) x_i\rangle -\langle \mu, vx_i\rangle = k \langle \alpha_{i_1}, vx_i \rangle.\] 
	The left hand side is an integer. In types $B_r,C_r$ and $D_r$, the possible pairings $\langle \alpha_{i_1}, vx_i \rangle = \langle v^{-1} \alpha_{i_1}, x_i \rangle \in \{-2,-1,0,1,2\}$. Again we can find a pair $v, x_i$ so that $\langle v^{-1}\alpha_{i_1},x_i\rangle$ is nonzero, in which case we see that $k \in \frac{1}{2} \mathbb{Z}$.

A similar argument applies to $\mu''$. 
\end{proof}

Note then that if $\mu'$ and $\mu''$ are distinct, then either $\mu'$ or $\mu'' \in \lambda+Q$, or, if both $\mu'$ and $\mu''$ are distinct half-integral weights, then there must exist some lattice point $\nu \in \mu''.\mu' \cap (\lambda+Q)$. We can summarize this as follows:

\begin{corollary}\label{segmentcor}Let $\mathfrak{g}$ and $\mu''.\mu'$ be as in Lemma \ref{typebcdhalfint}.
	If $\mu'' \neq \mu'$ then $\mu'_{int}$ exists.
\end{corollary}

\begin{remark}
	We do not see an easy way to obtain a similar corollary in the exceptional types; for instance in type $E_8$ we have, a priori, $\mu'=\mu+k' \alpha_{i_1}$ and $\mu''=\mu+k''\alpha_{i_1}$ for $k',k'' \in \frac{1}{6}\mathbb{Z}$, and there may be many ways for $\mu',\mu''$ to be distinct and for $\mu''.\mu' \cap (\lambda+Q) = \varnothing$. Thus this technique is not well-adapted to prove the saturation of Demazure polytopes in the exceptional types. Nonetheless we still expect saturation to hold, see our discussion in Section \ref{exceptionals}.
\end{remark}

Thus the remaining outstanding case to understand is the possibility that $\mu''=\mu' \in \lambda+\frac{1}{2}Q$ and not in $\lambda+Q$. However we will not approach this case directly, using induction instead. 

In order to prepare for the induction, we first prove some intermediate results analyzing what happens when $\op{dim} P_{\lambda}^{s_iw} < r$, where $r = \op{rank}(Q)$, the number of simple roots. 

\begin{lemma}
	Let $\lambda$ be a dominant weight and $w\in W$ any Weyl group element. Then $dim P_{\lambda}^w = r$ if and only if  $\op{supp}(\overline{w}) = \{1,\hdots,r\}$, where $\overline{w} \in W^{\lambda}$.
\end{lemma}

\begin{proof}
Note that $\dim P_\lambda^w \le r$ since $P_\lambda^w$ is contained in the hyperplane $\lambda + \Q(Q)$. 

	($\Rightarrow$) Suppose $i \notin \op{supp}(\bar w)$. Since $P_\lambda^{w} = P_\lambda^{\bar w}$, we may as well assume $w = \bar w$. Then $\langle u\lambda, x_i \rangle=\langle \lambda, x_i \rangle$ for every $u \leq w$. Thus the polytope $P_{\lambda}^w$ lies in the smaller hyperplane given by $\langle \lambda-y, x_i \rangle =0$.

($\Leftarrow$) Suppose $\op{supp}(\bar w) = \{1,\hdots,r\}$. Once again, let us assume $w = \bar w$. For each $i\in \{1,\hdots,r\}$, there exists some right (reduced) subword of $w$ that starts with $s_i$; in other words, there is a $v_i\le w$ such that $s_iv_i \le w$, with $v_i\to s_iv_i$. 

We claim that $v_i\lambda\ne s_iv_i\lambda$. Let $\hat v_i$ be the remaining left subword of $w$ such that $w = \hat v_i s_i v_i$ with $\ell(w) = \ell(\hat v_i)+ 1+\ell(v_i)$. If $v_i\lambda = s_iv_i\lambda$, then $w\lambda = \hat v_i v_i \lambda$, and $\ell(\hat v_i v_i)\le \ell(\hat v_i) + \ell(v_i)<\ell(w)$. This contradicts that $w$ is the min-length representative in the coset $wW_\lambda$. 

So $v_i\lambda - s_iv_i\lambda = \langle v_i\lambda, \alpha_i^\vee\rangle \alpha_i\ne 0$. Therefore $\alpha_i$ belongs to the linear span of $P_\lambda^w$. Since this holds for each $i$, $\dim P_\lambda^w \ge  r$.
\end{proof}	

Consider the hyperplanes defined by all $\mu$ such that $\langle \lambda, x_i \rangle = \langle \mu, x_i \rangle$. We call the faces given by these hyperplanes ``$\lambda$-adjacent".

\begin{lemma}\label{str}
	If $\op{dim} P_{w}^{\lambda}< \op{rank}(Q)$ and $\mu\in P_w^\lambda$, then $\mu$ is on a $\lambda$-adjacent face.
\end{lemma}

\begin{proof}
	Since $\op{dim} P_w^{\lambda} < \op{rank}(Q)$, we know that $\op{supp} ( \bar{w})$, and thus also $\op{supp} (\bar w^{-1})$, is not all of $\{1,\hdots,r\}$; here we use $\bar w$ to indicate the min-length representative of the coset $wW_\lambda$. For simplicity, assume $w = \bar w$. Let $x_i$ be the fundamental coweight corresponding to some index $i\notin \op{supp} (w^{-1})$. 
	
	Note the following useful  inequality among all inequalities defining this Demazure polytope: \[\langle \lambda, w^{-1} * e x_i \rangle \leq \langle \mu, ex_i \rangle. \]
		
	Since $i \notin \op{supp} (w^{-1})$ the first inequality becomes  \[\langle \lambda, x_i \rangle \leq \langle \mu, x_i \rangle. \] 
	Since $\mu\preceq \lambda$, we also have \[\langle \mu, x_i \rangle \leq \langle \lambda, x_i \rangle.  \] Stringing these inequalities together we obtain \[\langle \lambda, x_i \rangle \leq \langle \mu, x_i \rangle \leq \langle \lambda, x_i \rangle,\] and we see that any $\mu$ in $P^{\lambda}_w$ is necessarily contained in the face given by $\langle \lambda, x_i \rangle = \langle \mu, x_i \rangle$.
\end{proof}	

\begin{remark}
Another way to think of Lemma \ref{str} is that the faces corresponding to $(e,P_i)$ and $(w_0,P_\ell)$ coincide whenever $i\not\in \op{supp}(\bar w)$; here $\ell$ is the unique index such that $w_0x_i = -x_\ell$. 
\end{remark}

\begin{theorem} \label{BCD-result}
	Demazure characters in types $B_r,C_r,D_r$ are saturated.
\end{theorem}

\begin{proof}
	We prove this by double induction on $\op{rank}(Q)$ and length $\ell(w)$. The exact same argument applies to all three non-type-A classical types. Moreover, we will need to use that Demazure characters of type $A_r$ are saturated (Theorem \ref{A-result}). 
	
	For the base case, we merely observe that, in any type and rank, Demazure characters $\op{char}(V_{\lambda}^{e})$ are saturated (trivially, as $P_{\lambda}^e$ is the point $\lambda \in X^*(T)$), and Demazure characters $\op{char}(V_{\lambda}^{s_i})$ are also saturated, based on the $SL_2$ theory.
	
	Now assume that all Demazure characters for ranks $<r$ are saturated, and that in rank $r$, all Demazure characters $P_\lambda^w$ with $\ell(w)<k$ are saturated. 
 Let $\lambda$ be an arbitrary dominant weight and $w\in W$ an element such that $\ell(w) = k$, and choose a reduced word $w=s_{i_1} \dots s_{i_k}$. We consider two cases concerning $P_\lambda^w$. 
 \begin{enumerate}[label=(\alph*)]
 \item If $\op{dim} P_\lambda^w < r$, then by Lemma \ref{str}, there exists an index $i$ such that $i\not\in \op{supp}(\bar w)$. Let $L=L_i$ be the Levi subgroup of the maximal parabolic $P_i$. Clearly $V_{L,\lambda}^w = V_\lambda^w$ as $B_L$-modules (cf. Proposition \ref{facetDemazurestructure2}). Now $L$ has semisimple rank $r-1$ and each simple factor of $L$ is of type $A,B,C,$ or $D$. Saturation of $\op{char}(V_{L,\lambda}^w)$ follows from the saturation property of each of the simple factors (see subsequent, straightforward Lemma \ref{simple-factors}), so by induction $\op{char}(V_{L,\lambda}^w)$ and therefore $\op{char}(V_{\lambda}^w)$ are saturated. See Figure \ref{hey3}.~\\
 
 \tdplotsetmaincoords{70}{115}

\begin{figure}
	\begin{center}
		\begin{tikzpicture}[tdplot_main_coords,scale=.4]

				\draw[red,->] (-6,5,9) -- (8,-2,9);

		\coordinate(a) at (5,4,6);
		\coordinate(b) at (-5,9,6);
		\coordinate(c) at (9,-4,10);
		\coordinate(d) at (4,-9,15);
		\coordinate(e) at (-4,-5,15);
		\coordinate(f) at (-9,5,10);
		\coordinate(g) at (-6,5,9);
		\coordinate(h) at (-2,3,9);
		\coordinate(i) at (0,2,9);
		\coordinate(j) at (2,1,9);
		\coordinate(k) at (4,0,9);
		\coordinate(l) at (6,-1,9);
		\coordinate(m) at (8,-2,9);

		\draw[thick,fill=lightgray,opacity=0.5] (a) -- (b) -- (d) -- (c) -- (a);
		\draw[thick, dashed, black] (b) --(f) -- (e) -- (d);
		
		\node at (a) {\small $\bullet$};
		\node at (a) [label={[label distance=0.01mm]-90:e}] {};
		\node at (b) {\small $\bullet$};
		\node at (b) [label={[label distance=0.01mm]0:$s_1$}] {};
		\node at (c) {\small $\bullet$};
		\node at (c) [label={[label distance=0.01mm]180:$s_2$}] {};
		\node at (d) {\small $\bullet$};
		\node at (d) [label={[label distance=0.01mm]180:$s_2s_1$}] {};
		\node at (e) {\small $\bullet$};
		\node at (f) {\small $\bullet$};
		\node at (g) {\small $\bullet$};
		\node at (g) [label={[label distance=0.01mm]100:$\mu$}] {};
		\node at (h) {\small $\times$};
		\node at (h) [label={[label distance=0.01mm]-70:$\mu''$}] {};
		\node at (i) {\small $\bullet$};
		\node at (j) {\small $\bullet$};
		\node at (k) {\small $\bullet$};
		\node at (l) {\small $\bullet$};
		\node at (m) {\small $\times$};
		\node at (m) [label={[label distance=0.01mm]-90:$\mu'$}] {};

				\end{tikzpicture}
		\caption{\label{hey3} \small In type $A_3$, the polytope $P_{5\varpi_1+4\varpi_2+6\varpi_3}^{s_2s_1}$. Root lines are shown in red for some weights $\mu \in P_{5\varpi_1+4\varpi_2+6\varpi_3}^{s_1s_2s_1}$ but not in $P_{5\varpi_1+4\varpi_2+6\varpi_3}^{s_2s_1}$. In this example $\mu=-6\varpi_1+5\varpi_2+9\varpi_3$, $\mu''=-2\varpi_1+3\varpi_2+9\varpi_3$ and $\mu'=8\varpi_1-2\varpi_2+9\varpi_3$. }
	\end{center}
\end{figure}
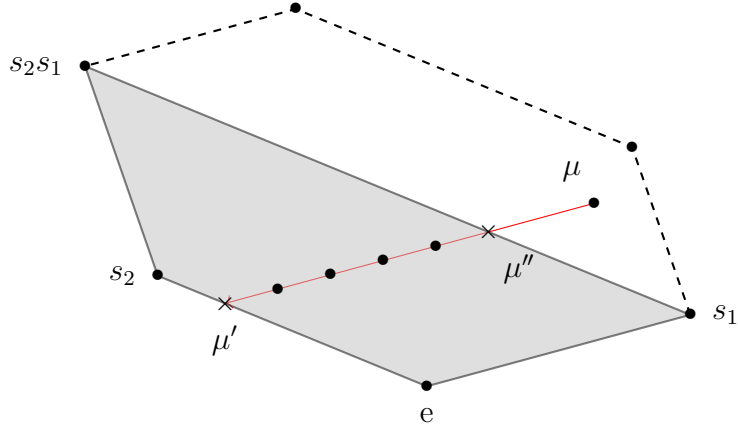
 
 \item Otherwise, $\dim P_\lambda^w = r$. Take $\mu, \mu', \mu''$ as in Lemma \ref{typebcdhalfint}. We consider two further cases: ~\\
 \begin{enumerate}[label=(\roman*)]
 \item If $\dim P_\lambda^{s_{i_1}w}<r$, then in particular $\dim P_\lambda^{s_{i_1}w} = r-1$ (the dimension can only increase by one passing from $P_\lambda^{s_{i_1}w}$ to $P_\lambda^w$.) By Lemma $\ref{str}$, there is a unique index $i$ such that $i\not\in \op{supp}(\overline{s_{i_1} w})$, and for this index, 
 \setcounter{theorem}{\value{theorem}+1}
\begin{align}\label{eq'n-2-use}
\langle \lambda, x_i\rangle = \langle \mu_1, x_i\rangle
\end{align}
for all $\mu_1\in P_\lambda^{s_{i_1}w}$. Since $\dim P_\lambda^w = r$, $i\in \op{supp}(\bar w)$. But the change from $\op{supp}(\overline{s_{i_1}w})$ to $\op{supp}(\bar w)$ can only be in the index $i_1$. Therefore $i = i_1$. See Figure \ref{hey4}. ~\\
 \tdplotsetmaincoords{70}{115}

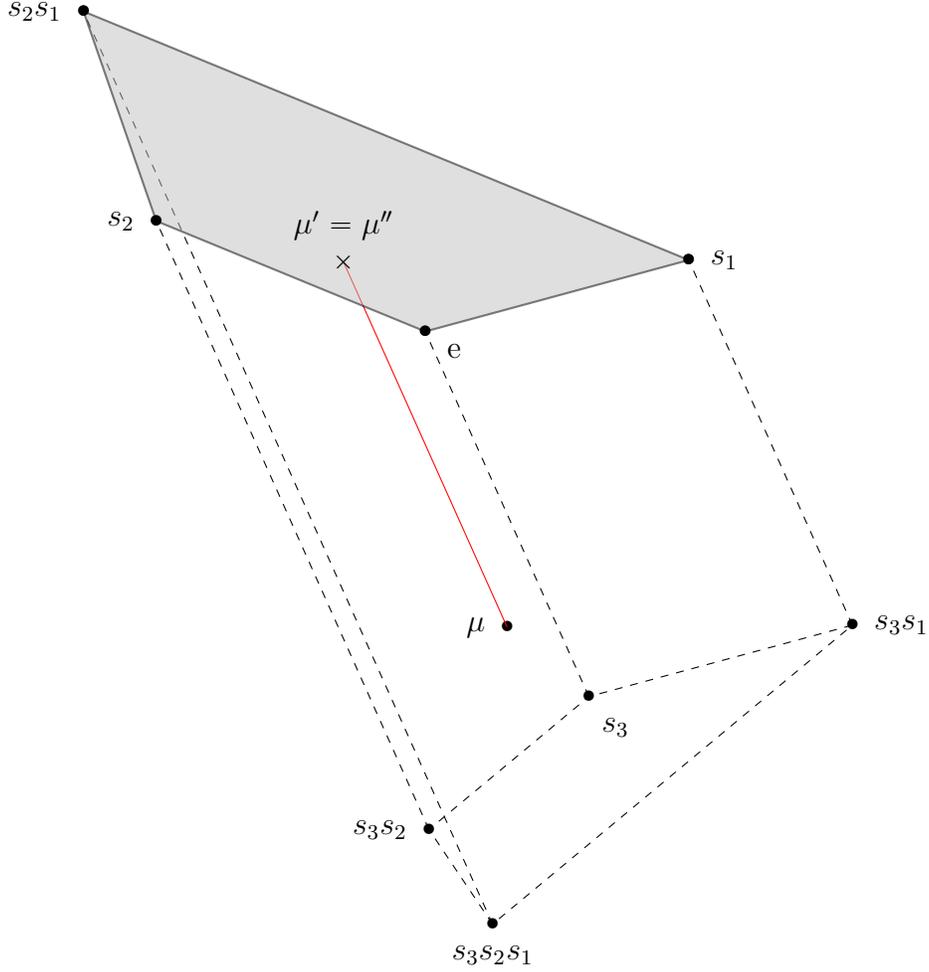
\begin{figure}
	\begin{center}
		\begin{tikzpicture}[tdplot_main_coords,scale=.4]
		\draw[dashed,black] (5,10,-6) -- (5,4,6);
		\draw[dashed,black] (-5,15,-6) -- (-5,9,6);
		\draw[dashed,black] (4,6,-15) -- (4,-9,15);
		\draw[dashed,black] (9,6,-10) -- (9,-4,10);

			\node at (a) {\small $\bullet$};
		\node at (a) [label={[label distance=0.01mm]-10:e}] {};
		\node at (b) {\small $\bullet$};
		\node at (b) [label={[label distance=0.01mm]0:$s_1$}] {};
		\node at (c) {\small $\bullet$};
		\node at (c) [label={[label distance=0.01mm]180:$s_2$}] {};
		\node at (d) {\small $\bullet$};
		\node at (d) [label={[label distance=0.01mm]180:$s_2s_1$}] {};
		\coordinate(e) at (5,10,-6);
		\node at (e) [label={[label distance=0.01mm]-80:$s_3$}] {};
		\coordinate(f) at (-5,15,-6);
		\node at (f) [label={[label distance=0.01mm]0:$s_3s_1$}] {};
		\coordinate(g) at (4,6,-15);
		\node at (g) [label={[label distance=0.01mm]-90:$s_3s_2s_1$}] {};
		\coordinate (h) at (9,6,-10);
		\node at (h) [label={[label distance=0.01mm]180:$s_3s_2$}] {};
		\coordinate (i) at (5,7,-4);
		\node at (i) {\small $\bullet$};
		\node at (i) [label={[label distance=0.01mm]180:$\mu$}] {};
		\coordinate (j) at (5,1,8);

		\draw[dashed,black] (g) --  (h) -- (e) -- (f) -- (g);
		\draw[red] (i) -- (j);

		\draw[thick,fill=lightgray,opacity=0.5] (a) -- (b) -- (d) -- (c) -- (a);

		\node at (a) {\small $\bullet$};
		\node at (b) {\small $\bullet$};
		\node at (c) {\small $\bullet$};
		\node at (d) {\small $\bullet$};
		\node at (e) {\small $\bullet$};
		\node at (f) {\small $\bullet$};
		\node at (g) {\small $\bullet$};
		\node at (h) {\small $\bullet$};

				\node at (j) {\small $\times$};
		\node at (j) [label={[label distance=0.01mm]90:$\mu'=\mu''$}] {};
		
		\end{tikzpicture}
		\caption{\label{hey4} \small In type $A_3$, the polytope $P_{5 \varpi_1+4\varpi_2+6\varpi_3}^{s_2s_1}$. The polytope for $P_{5 \varpi_1+4\varpi_2+6\varpi_3}^{s_3s_2s_1}$ is shown with dashed lines. An $\alpha_3$ root line is shown in red joining some weights in $P_{5 \varpi_1+4\varpi_2+6\varpi_3}^{s_3s_2s_1}$ but not in $P_{5 \varpi_1+4\varpi_2+6\varpi_3}^{s_2s_1}$. Here $\mu=5\varpi_1+7\varpi_2-4\varpi_3$ and $\mu'=\mu''=5\varpi_1+\varpi_2+8\varpi_3$.}
	\end{center}
\end{figure}

Examining (\ref{eq'n-2-use}) with $\mu_1 = \mu'$, we see that 
\begin{align*}
\langle \lambda,x_i\rangle &= \langle \mu, x_i\rangle + k' \langle \alpha_i, x_i\rangle \\
\langle \lambda - \mu, x_i\rangle &= k',
\end{align*}
so since $\lambda-\mu\in Q$, $k'$ is an integer and $\mu_{int}' = \mu'$ exists inside $\mu''.\mu'$ (moreover, $\mu'' = \mu'$). By inductive assumption, $V_\lambda^{s_{i_1}w}(\mu') \ne0$; hence $V_\lambda^w (\mu) \ne0$ by Proposition \ref{indsketch}. ~\\

 \item Otherwise, $\dim P_\lambda^{s_{i_1}w} = r$ as well. We distinguish two final possibilities: either $\mu' \neq \mu''$ or $\mu'=\mu''$. See Figure \ref{hey5}. 
  \tdplotsetmaincoords{70}{115}

\begin{figure}
	\begin{center}
		\begin{tikzpicture}[tdplot_main_coords,scale=.4]

	\node at (a) {\small $\bullet$};

	\node at (b) {\small $\bullet$};
	\node at (b) [label={[label distance=0.01mm]0:$s_1$}] {};
	\node at (c) {\small $\bullet$};
	\node at (c) [label={[label distance=0.01mm]180:$s_2$}] {};
	\node at (d) {\small $\bullet$};
	\node at (d) [label={[label distance=0.01mm]180:$s_2s_1$}] {};
	\coordinate(e) at (5,10,-6);
	
	\coordinate(f) at (-5,15,-6);
	\node at (f) [label={[label distance=0.01mm]0:$s_3s_1$}] {};
	\coordinate(g) at (4,6,-15);
	\node at (g) [label={[label distance=0.01mm]-90:$s_3s_2s_1$}] {};
	\coordinate (h) at (9,6,-10);

		\coordinate (l) at (4,10,-7);
		\coordinate (m) at (15,-10,4);
		\node at (m) {\small $\bullet$};
		\node at (m) [label={[label distance=0.01mm]180:$s_2s_3$}] {};
		\coordinate (n) at (15,-6,-4);
		\node at (n) {\small $\bullet$};
		\node at (n) [label={[label distance=0.01mm]180:$s_2s_3s_2$}] {};
		\coordinate (o) at (10,-6,-9);
		\node at (o) {\small $\bullet$};
		\node at (o) [label={[label distance=0.01mm]180:$s_2s_3s_2s_1$}] {};
		\coordinate (p) at (10,-15,9);
		\node at (p) {\small $\bullet$};
		\node at (p) [label={[label distance=0.01mm]180:$s_2s_3s_1$}] {};

		\draw[thick,fill=lightgray,opacity=0.5] (e) -- (f) -- (g) -- (h) -- (e);
		\draw[thick,fill=lightgray,opacity=0.2] (e) -- (f) -- (g) -- (h) -- (e);
		\draw[thick,fill=lightgray,opacity=0.8] (d) -- (g) -- (h) -- (c) -- (d);
		\draw[thick,fill=lightgray,opacity=0.4] (a) -- (b) -- (f) -- (e) -- (a);
		\draw[thick,fill=lightgray,opacity=0.3] (a) -- (e) -- (h) -- (c) -- (a);
		\draw[thick,fill=lightgray,opacity=0.5] (a) -- (b) -- (d) -- (c) -- (a);
		\draw[dashed] (d) -- (c) -- (m) -- (p) -- (d);
		\draw[dashed] (h) -- (g) -- (o) -- (n) -- (h);
		\draw[dashed] (m) -- (n);
		\draw[thick,fill=lightgray,opacity=0.3] (m) -- (p) -- (o) -- (n) -- (m);
		\coordinate (q) at (10,-2,-9);
		\node at (q) {\small $\bullet$};
		\node at (q) [label={[label distance=0.01mm]180:$\mu$}] {};
		\coordinate (r) at (6,6,-13);
		\node at (r) {\small $\times$};
		\node at (r) [label={[label distance=0.01mm]0:$\mu'$}] {};
		\coordinate (s) at (10,-6,-1);
		\node at (s) {\small $\bullet$};
		\node at (s) [label={[label distance=0.01mm]180:$\nu$}] {};
		\coordinate (t) at (6,2,-5);
		\node at (t) {\small $\times$};
		\node at (t) [label={[label distance=0.01mm]210:$\nu''$}] {};
		\coordinate (u) at (2,10,-9);
		\node at (u) {\small $\times$};
		\node at (u) [label={[label distance=0.01mm]30:$\nu'$}] {};
		\draw[red] (q) -- (r);
		\draw[red] (s) -- (t);
		\draw[red,dashed] (t) -- (u);
		
		\node at (e) [label={[label distance=0.01mm]80:$s_3$}] {};
		\node at (a) [label={[label distance=0.01mm]-10:e}] {};
		\node at (h) [label={[label distance=0.03mm]0:$s_3s_2$}] {};

		\node at (a) {\small $\bullet$};
		\node at (b) {\small $\bullet$}; 
		\node at (c) {\small $\bullet$};
		\node at (d) {\small $\bullet$};
		\node at (e) {\small $\bullet$};
		\node at (f) {\small $\bullet$};
		\node at (g) {\small $\bullet$};
		\node at (h) {\small $\bullet$};

		\end{tikzpicture}
		\caption{\label{hey5} \small In type $A_3$, the polytope $P_{5\varpi_1+4\varpi_2+6\varpi_3}^{s_3s_2s_1}$. The polytope for $P_{5\varpi_1+4\varpi_2+6\varpi_3}^{s_2s_3s_2s_1}$ is shown with the dashed edges. Some root lines in the $\alpha_2$ direction are shown in red, associated to weights $\mu$ in $P_{5\varpi_1+4\varpi_2+6\varpi_3}^{s_2s_3s_2s_1}$ but not in $P_{5\varpi_1+4\varpi_2+6\varpi_3}^{s_3s_2s_1}$. In particular, here $\mu=10\varpi_1-2\varpi_2-9\varpi_3$, and $\mu''=\mu'=6\varpi_1+6\varpi_2-13\varpi_3$. Since the dimension of both polytopes is 3, we observe that if $\mu''=\mu'$ then $\mu'$ and $\mu$ both lie on a face of $P_{5\varpi_1+4\varpi_2+6\varpi_3}^{s_2s_3s_2s_1}$. On the other hand, we have $\nu=10\varpi_1-6\varpi_2-\varpi_3$ in the interior of $P_{5\varpi_1+4\varpi_2+6\varpi_3}^{s_2s_3s_2s_1}$, and thus $\nu''=6\varpi_1+2\varpi_2-5\varpi_3 \neq \nu'=2\varpi_2+10\varpi_2-9\varpi_3$. }
	\end{center}
\end{figure}
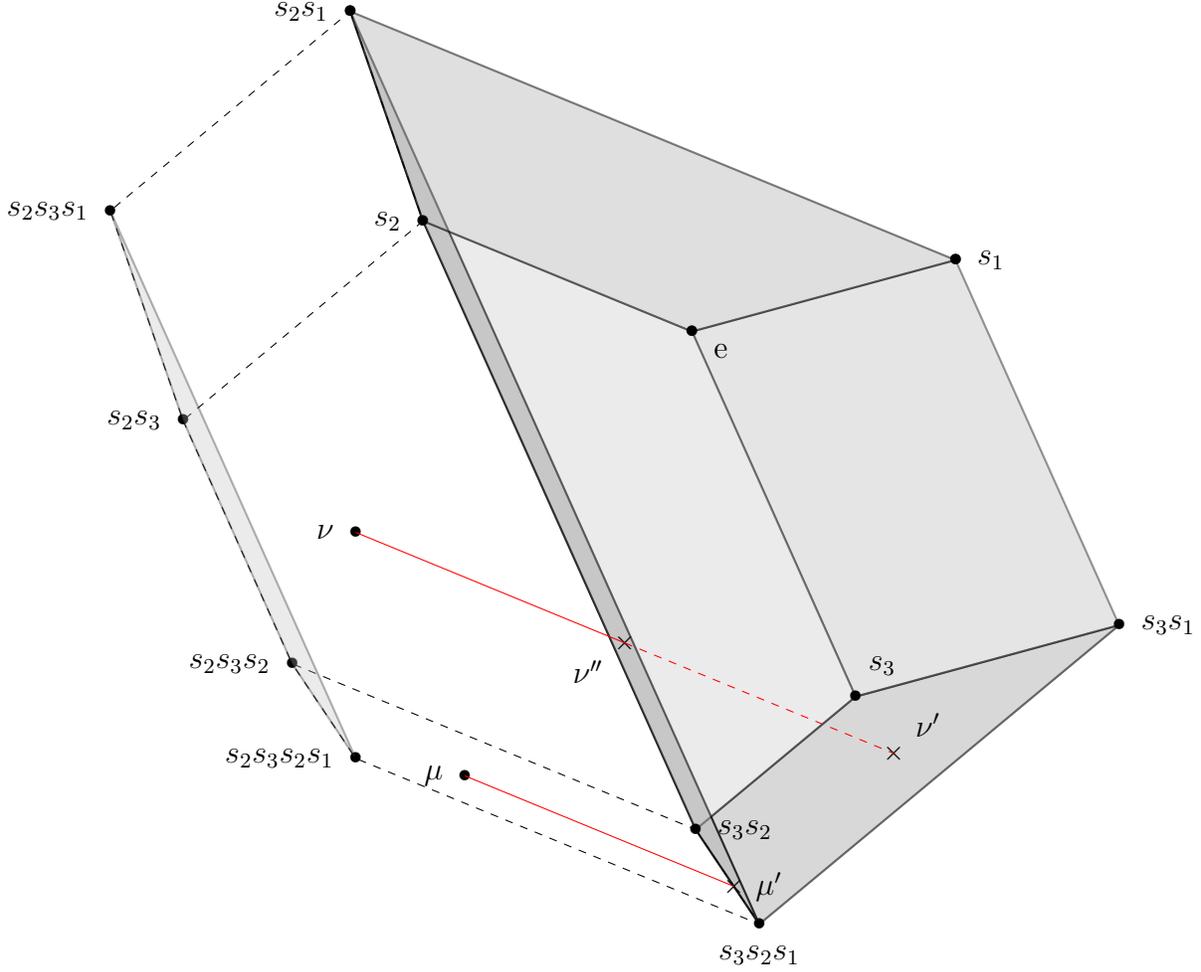
 The first case is handled by Corollary \ref{segmentcor} and Proposition \ref{indsketch}. 
 
	Now suppose $\mu'=\mu''$. Set
	$$
F(\mu') = \{(v,P_i): v\in W^{P_i}, \langle \lambda, (s_{i_1}w)^{-1}*v x_i\rangle = \langle \mu', v x_i\rangle\}, 
	$$
	which records the faces of $P_\lambda^{s_{i_1}w}$ that contain $\mu'$. 
	Since $\mu'+\epsilon \alpha_{i_1}\not\in P_\lambda^{s_{i_1}w}$ for any $\epsilon>0$, there exists some pair $(v,P_i)\in F(\mu')$ such that $\langle \alpha_{i_1},v x_i\rangle<0$. Likewise, by the minimality of $\mu''$, there is a pair $(u,P_j)\in F(\mu')$ such that $\langle \alpha_{i_1}, ux_j \rangle >0$. ~\\
	
	Let $a,b$ be positive constants such that $a\langle \alpha_{i_1}, vx_i\rangle + b \langle \alpha_{i_1},ux_j\rangle = 0$. Observe that, for every $\mu_1\in P_{\lambda}^{s_{i_1}w}$, the inequality 
	\setcounter{theorem}{\value{theorem}+1}
	\begin{align}\label{make-a-face}
	a\langle \lambda , (s_{i_1}w)^{-1}*v x_i\rangle+b \langle \lambda , (s_{i_1}w)^{-1}*u x_j\rangle \le \langle \mu_1, a vx_i+bux_j\rangle
	\end{align}
	holds. Moreover, it holds on $P_\lambda^w$: for if $\mu_2\in P_\lambda^w$, then one can find $t\in \Q$ such that $\mu_2+t \alpha_{i_1}\in P_\lambda^{s_{i_1}w}$. Then by construction
	$$
	\langle \mu_2, avx_i+bux_j\rangle = \langle \mu_2+t\alpha_{i_1}, avx_i+bux_j\rangle,
	$$
	so since $\mu_2+t\alpha_{i_1}$ satisfies (\ref{make-a-face}), so does $\mu_2$. 
	~\\
	
	Finally, since $\mu'$ satisfies (\ref{make-a-face}) with equality, so does $\mu$. Thus $\mu$ lies on a face $\mathcal{F}'$ of dimension $\leq r-1$ of the polytope $P_{\lambda}^w$. 
	Now we use Proposition \ref{facetDemazurestructure2} to assert that faces of Demazure modules have the structure of (conjugated) Demazure modules for Levi subgroups. So by the inductive assumption on rank, and Lemma \ref{simple-factors}, we see that $\mathcal{F}'$ is saturated, so $\mu$ is a nontrivial weight of $V_{\lambda}^w$. Thus saturation holds for $P_{\lambda}^w$. \qedhere	
 \end{enumerate}
 \end{enumerate}
\end{proof}

\begin{lemma}\label{simple-factors}
Suppose $\mathfrak{g} = \mathfrak{g}_1\oplus \mathfrak{g}_2$, with each $\mathfrak{g}_i$ a reductive Lie algebra with chosen Borel and maximal toral subalgebras $\mathfrak{h}_i\subseteq \mathfrak{b}_i$. Set $\mathfrak{b} = \mathfrak{b}_1\oplus \mathfrak{b}_2$ and $\mathfrak{h} = \mathfrak{h}_1\oplus \mathfrak{h}_2$. 

Each dominant weight $\lambda$ in $\mathfrak{h}^*$ is uniquely a sum of dominant weights $\lambda_1+\lambda_2$ coming from $\mathfrak{h}_i^*$. Each Weyl group element $w$ for $\mathfrak{g}$ can be decomposed uniquely as $w_1w_2 = w_2w_1$ where $w_i$ is a Weyl group element for $\mathfrak{g}_i$. 

Demazure modules of $\mathfrak{g}, \mathfrak{g}_1,\mathfrak{g}_2$ are related as follows:
\begin{enumerate}
\item For each $\lambda = \lambda_1+\lambda_2$ and $w = w_1w_2$, we have an isomorphism of $\mathfrak{b}$-modules
$$
V_\lambda^w = V_{\lambda_1}^{w_1}\otimes V_{\lambda_2}^{w_2}.
$$
\item Therefore $\op{char}(V_\lambda^w) = \op{char}(V_{\lambda_1}^{w_1})\cdot \op{char}(V_{\lambda_2}^{w_2})$, 
\item $P_\lambda^w = P_{\lambda_1}^{w_1}\times P_{\lambda_2}^{w_2}$, and 
\item $\op{char}(V_\lambda^w)$ is saturated if and only if $\op{char}(V_{\lambda_1}^{w_1})$ and $\op{char}(V_{\lambda_2}^{w_2})$ are saturated. 
\end{enumerate}
\end{lemma}	

\section{Exceptional Types and Calculations}\label{exceptionals}

In this section, we prove Theorem \ref{lattice} for types $F_4$ and $G_2$ using methods quite different from the preceding sections. Because of the low ranks, we are able to accomplish this by computer-assisted calculations. For type $E$, we conjecture that the saturation property will also hold. 

\subsection{Computational approach in the exceptional types}

Let $G$ be of any simple Lie type. Fix $w\in W$. Consider the cone 
$$
\mathcal{C}_w = \{(\lambda,\mu)\in (X^*(T)_\Q)^2 : \lambda \text{ is a dominant rational weight and }\mu \in P_\lambda^w\}\subseteq (X^*(T)_\Q)^2.
$$
Then $\mathcal{C}_w$ is a rational polyhedral cone inside $(X^*(T)_\Q)^2$. Its extremal rays are generated by pairs $(\omega_i, v\omega_i)$, such that $\omega_i$ is a fundamental weight and $v\le w$. Its inequalities are those of Theorem \ref{necsufineqs}. The statement of Theorem \ref{lattice} can be rephrased as follows:

\begin{theorem}\label{case-by-case}
Suppose that $G$ is of type $A_r, B_r, C_r, D_r, F_4$, or $G_2$. Then for dominant integral $\lambda$, $V_\lambda^w(\mu)\ne 0$ if and only if $\lambda-\mu\in Q$ and $(\lambda,\mu)\in \mathcal{C}_w$. 
\end{theorem}

This is a statement about every lattice point inside $\mathcal{C}_w$, assuming we take the ambient lattice to be 
$$
\Lambda = \{(\lambda,\mu)\in (X^*(T))^2: \lambda-\mu\in Q\}.
$$
Let $\mathcal{S}_w = \mathcal{C}_w\cap \Lambda$. Then $\mathcal{S}_w$ is the semigroup of lattice points inside a rational polyhedral cone. By the general theory \cite{Schrijv}*{Theorem 16.4}, $\mathcal{S}_w$ has a finite minimal generating set, normally called its \emph{Hilbert basis}. 

\begin{lemma} \label{additive}
Suppose that for every $(\lambda,\mu)$ in the Hilbert basis of $\mathcal{S}_w$, $V_\lambda^w(\mu)\ne 0$. Then for every $(\lambda,\mu)\in \mathcal{S}_w$, $V_\lambda^w(\mu)\ne 0$. 
\end{lemma}

\begin{proof}
It suffices to show that the property $V_\lambda^w(\mu)\ne 0$ is closed under addition. Suppose that $V_{\lambda_1}^w(\mu_1)\ne 0$ and $V_{\lambda_2}^w(\mu_2)\ne 0$; we claim that $V_{\lambda_1+\lambda_2}^w(\mu_1+\mu_2)\ne0$. 

Recalling the discussion of the Borel-Weil theorem prior to Lemma \ref{borel-weil}, for each $i=1,2$ we have $H^0(X_w, \C_{\mu_i}\otimes L_w(\lambda_i))^T\ne0$. Let $\sigma_i$ be a nonzero $T$-invariant global section in $H^0(X_w, \C_{\mu_i}\otimes L_w(\lambda_i))$. Then $\sigma_1\otimes \sigma_2$ is nonzero (since $X_w$ is irreducible) and belongs to
$$
H^0(X_w, \C_{\mu_1+\mu_2}\otimes L_w(\lambda_1+\lambda_2))^T.
$$
Therefore $V_{\lambda_1+\lambda_2}^w(\mu_1+\mu_2)\ne0$. 
\end{proof}

So in order to verify Theorem \ref{case-by-case} for any particular group $G$ and Weyl group element $w\in W$, one can in principle follow this algorithm: 
\begin{enumerate}
\item Find the Hilbert basis $\{(\lambda_1,\mu_1), \hdots, (\lambda_t,\mu_t)\}$ of the semigroup $\mathcal{S}_w$.
\item Verify that $V_{\lambda_i}^w(\mu_i)\ne 0$ for each $i=1,\hdots,t$. 
\end{enumerate}

As long as the weights $\lambda_i$ are not too large, step (2) can be done with Demazure's character formula. Specialized computer programs such as {\tt LiE} \cite{lie} are useful here. 

The more challenging step seems to be (1), but actually this is a standard problem in linear programming: to start from a rational, polyhedral cone $\mathcal{C}$ with known inequalities and produce the Hilbert basis of $\mathcal{C}\cap \Lambda$. Several computer programs are available, such as {\tt Normaliz} \cite{nmz}. 

Before describing some of the computational results below, we remark that, based on computer-generated examples in various types, the Hilbert basis of $\mathcal{S}_w$ seems to consist entirely of pairs $(\lambda, \mu)$ such that $\lambda$ is a fundamental weight.  Let us call this property $\mathcal{P}$. 
$$
\mathcal{P}:~~\text{ the Hilbert basis of $\mathcal{S}_w$ is exactly $\{(\omega_j,\mu): 1\le j\le n, V_{\omega_j}^w(\mu)\ne 0\}$} 
$$
Clearly the Hilbert basis must contain the above set, so property $\mathcal{P}$ is asserting that the Hilbert basis of $\mathcal{S}_w$ is as small as possible. In fact, for any $G$ and $w\in W$, the validity of property $\mathcal{P}$ is equivalent to the validity of Theorem \ref{case-by-case} or equivalently Theorem \ref{lattice} for $w$.

\begin{proposition} \label{PropP} Let $w \in W$. Then property $\mathcal{P}$ holds for $w$ if and only if, for any $(\lambda, \mu) \in \mathcal{S}_w$, $V_\lambda^w(\mu) \neq 0$. 
\end{proposition} 

\begin{proof} 
($\Rightarrow$) Suppose property $\mathcal{P}$ holds, and let $(\lambda, \mu) \in \mathcal{S}_w$. Since each Hilbert basis element is of the form $(\omega_i, \nu)$ with $V_{\omega_i}^w(\nu) \neq 0$, and $(\lambda, \mu)$ can be expressed as a positive integer sum of such pairs, by the additivity Lemma \ref{additive}, we immediately conclude that $V_\lambda^w(\mu) \neq 0$.

($\Leftarrow$) Conversely, suppose that for any $(\lambda, \mu) \in \mathcal{S}_w$ we have that $V_\lambda^w(\mu) \neq 0$. Write $\lambda=\sum_{i \in I} \omega_i$ as a sum of fundamental weights over a set of (not necessarily distinct) indices $I$. Then there is a $\mathfrak{g}$-module inclusion 
$$
V(\lambda) \hookrightarrow \bigotimes_{i \in I} V(\omega_i),
$$
sending the highest weight vector $v_\lambda \in V_\lambda$ to $\bigotimes_{i \in I} v_{\omega_i}$. The extremal weight vector $v_{w\lambda}$ can be expressed as $\bigotimes_{i \in I} v_{w \omega_i}$ for appropriate extremal weight vectors $v_{w \omega_i} \in V_{\omega_i}(w \omega_i)$ for each $i \in I$. Since the inclusion is a $\mathfrak{g}$-module map, we can get a realization of the Demazure module $V_\lambda^w$ via $U(\mathfrak{b}). \left(\bigotimes_{i \in I} v_{w \omega_i} \right)$.

Now, since $V_\lambda^w(\mu) \neq 0$, we can find a nonzero vector of weight $\mu$ of the form $E.v_{w\lambda}$ for some $E \in U(\mathfrak{b})$. Consider the image 
$$
E.\left(\bigotimes_{i \in I} v_{w \omega_i} \right)
$$
under the inclusion. This, in general, will be a sum of many tensor factors. But, since we know that it is nonzero, select any nonzero summand, which will be of the form 
$$
\bigotimes_{i \in I} E_i. v_{w \omega_i},
$$
with each $E_i \in U(\mathfrak{b})$. This vector still has total weight $\mu$, and each $E_i.v_{w\omega_i} \in V_{\omega_i}^w(\nu_i)$ for some nontrivial weight $\nu_i$ of the fundamental Demazure module $V_{\omega_i}^w$. Thus, we can decompose 
$$
(\lambda, \mu) = \sum_{i \in I} (\omega_i, \nu_i);
$$
since such pairs $(\omega_i, \nu_i)$ are necessarily in the Hilbert basis and by the above argument suffice to express any $(\lambda, \mu) \in S_w$ with $V_\lambda^w(\mu) \neq 0$, these pairs precisely form the Hilbert basis of $\mathcal{S}_w$. Thus, property $\mathcal{P}$ holds. 
\end{proof}

\subsection{Type $F_4$}

The Weyl group of type $F_4$ has order $|W| = 1152$.  Using {\tt Normaliz}, we found the Hilbert basis $\mathcal{H}_w$ of each one of the 1152 semigroups $\mathcal{S}_w$. In every instance, $\mathcal{H}_w$ has property $\mathcal{P}$. 
Then, using {\tt LiE}, for each $w\in W$ we found that $(\lambda,\mu)\in \mathcal{H}_w\implies V_\lambda^w(\mu)\ne0$. This confirms Theorem \ref{case-by-case} for type $F_4$. 

Here is a sample of the data we produced. We follow the conventions of Bourbaki \cite{Bour}. For instance, with $w = s_1s_2s_3s_4$, both the Hilbert basis and set of (primitive elements on) extremal rays coincide with the following set of $14$ vectors. 
$$
\begin{array}{|c|c|}\hline
 (\lambda,\mu) & (\lambda,\mu)  \\ \hline
 (\omega_4,-\omega_1+\omega_3) & (\omega_4,\omega_4) \\ 
 (\omega_4,\omega_3-\omega_4) & (\omega_4,\omega_2-\omega_3) \\ 
 (\omega_4, \omega_1 - \omega_2 + \omega_3) & (\omega_3, -\omega_1+\omega_3+\omega_4) \\
 (\omega_3,\omega_3) & (\omega_3, \omega_2 - \omega_3+\omega_4) \\ 
 (\omega_3, \omega_1 - \omega_2+\omega_3+\omega_4 ) & ( \omega_2, -\omega_1+2\omega_3) \\
 (\omega_2,\omega_2) & (\omega_2, \omega_1 - \omega_2+2\omega_3)\\
 (\omega_1,-\omega_1+\omega_2) & (\omega_1,\omega_1) \\\hline
 \end{array}
$$

\subsection{Type $G_2$}

Once again we follow the conventions of Bourbaki \cite{Bour}. The Weyl group of type $G_2$ has order $|W| = 12$. Using {\tt Normaliz}, we found the Hilbert basis $\mathcal{H}_w$ for each of the 12 semigroups $\mathcal{S}_w$. Each $\mathcal{H}_w$ has property $\mathcal{P}$. Since space allows, we list the numbers of Hilbert basis elements and extremal rays for each $\mathcal{S}_w$. 

$$
\begin{array}{|l|c|c|}\hline
w & \text{\# Hilbert basis elements} &\text{\# Extremal rays} \\\hline
e & 2 &  2 \\
s_1 & 3 & 3 \\ 
s_2 & 3 & 3  \\
s_1s_2 & 7 & 5 \\
s_2s_1 & 5 & 5 \\
s_1s_2s_1 & 10 & 7 \\
s_2s_1s_2 & 11 & 7 \\
s_1s_2s_1s_2 & 17 & 9 \\
s_2s_1s_2s_1 & 14 & 9 \\
s_1s_2s_1s_2s_1 & 19 & 11 \\
s_2s_1s_2s_1s_2 & 19 & 11 \\
w_0 & 20 & 12 \\\hline
\end{array}
$$
~\\
Using {\tt LiE}, for each $w\in W$ we found that $(\lambda,\mu)\in \mathcal{H}_w\implies V_\lambda^w(\mu)\ne0$. This confirms Theorem \ref{case-by-case} for type $G_2$.

\subsection{Types $E_6$, $E_7$ and $E_8$}

Since $E_6$ and $E_7$ appear as Levi subtypes to $E_8$, the extension of Theorem \ref{case-by-case} to type $E$ could in theory be established by applying the above algorithm to type $E_8$. However, this is a formidable task due to 
$$
|W_{E_8}| = 696729600.
$$
Even for type $E_6$ we have $|W_{E_6}|=51840$, a significant undertaking. 
For the time being, our methods do not cover the case of type $E$, and we leave it as a conjecture. 
\begin{conjecture}
If $G$ is simple of type $E$, then for any $w\in W$ and dominant $\lambda$, then $\op{char}(V_\lambda^w)$ is saturated. 
\end{conjecture}

\section{Statements and Declarations}

On behalf of all authors, the corresponding author states that there is no conflict of interest.

Our manuscript has no associated data.



\end{document}